\documentclass[leqno,11pt]{amsart}

\usepackage{palatino}
\usepackage[mathcal]{euler}

\usepackage{graphicx}
\usepackage{epstopdf,epsfig}
\usepackage{tikz-cd}
\usepackage[font=small]{caption}
\usepackage{xcolor}
\usepackage{url}

\usepackage{amsmath,amsfonts,amsthm,amssymb,slashed}

\usepackage{MnSymbol}
\allowdisplaybreaks

\usepackage[all]{xy}
\xyoption{arc}
 
\usepackage{enumerate}

\numberwithin{equation}{subsection}

 \swapnumbers

\theoremstyle{plain}
\newtheorem*{theorem*}{Theorem} 
\newtheorem*{proposition*}{Proposition} \newtheorem*{lemma*}{Lemma}
\newtheorem*{assumption*}{Assumption}
\newtheorem*{conjecture*}{Conjecture}

\newtheorem{theorem}[equation]{Theorem} 
\newtheorem{lemma}[equation]{Lemma}
\newtheorem{corollary}[equation]{Corollary}
\newtheorem{proposition}[equation]{Proposition}

\theoremstyle{definition}
\newtheorem{definition}[equation]{Definition}

\newtheorem{remark}[equation]{Remark}

\newtheorem*{remark*}{Remark}
\newtheorem*{definition*}{Definition}
\newtheorem*{remarks*}{Remarks}
\newtheorem*{observation*}{Observation}

\theoremstyle{remark}

\newcommand{\Compact}{\mathfrak{K}}

\newcommand{\R}{\mathbb{R}}
\newcommand{\C}{\mathbb{C}}
\newcommand{\Z}{\mathbb{Z}}

\newcommand{\GG}{\pmb{G}}
\newcommand{\KK}{\pmb{K}}

\newcommand{\tempiric}{\mathrm{tempiric}}

\newcommand{\PSDO}{\mathsf{P}}

\newcommand{\CC}{\mathsf{C}}
\newcommand{\Symb}{\mathsf{S}}
\newcommand{\Rep}{\mathsf{Rep}}
\newcommand{\Fin}{\mathsf{Fin}}

 \DeclareMathOperator{\Hom}{Hom}

\DeclareMathOperator{\Supp}{Supp}

\DeclareMathOperator{\End}{End}

\newcommand{\CoKa}{{\textsc{\textsf{ck}}}\,}
\newcommand{\mult}{{\textsc{\textsf{mult}}}}
\newcommand{\proj}{{\textsc{\textsf{proj}}}}

\usepackage[hidelinks]{hyperref}

\usepackage{cancel} 
\usepackage[normalem]{ulem}

\begin{document}

\title[Pseudodifferential Operators and Connes-Kasparov]{Pseudodifferential Operators and the Connes-Kasparov Isomorphism} 
\author{Peter Debello and Nigel Higson}
\address{Department of Mathematics, Penn State University, University Park, PA 16802, USA}

\date{\today}

\begin{abstract}
We compute the $K$-theory of the $C^*$-category generated by order zero, equivariant, properly supported, classical  pseudodifferential  operators acting on sections of homogeneous bundles over the symmetric space  of a real reductive Lie group $G$. Our result uses the Connes-Kasparov isomorphism for $G$, and in fact it is equivalent to the Connes-Kasparov isomorphism.  We relate our computation to David Vogan's well-known param\-etrization of the tempered irreducible representations of $G$ with real infinitesimal character.  When the reductive group $G$ has real rank one, we formulate and prove a Fourier isomorphism theorem for equivariant order zero pseudodifferential operators on the symmetric space, and use it to prove a $K$-theoretic version of Vogan's theorem. 
\end{abstract}

\maketitle

\section{Introduction} 
In the 1980's Alain Connes (see \cite[Conj.~4.1]{Rosenberg84}) and Gennadi Kasparov \cite[Sec.~5, Conj.~1] {KasparovICM83}  suggested a means of describing the $K$-theory groups of the reduced $C^*$-algebra of an almost-connected Lie group using the index theory of Dirac-type operators.  
Their conjecture, now verified, is called the \emph{Connes-Kasparov isomorphism} in $C^*$-algebra $K$-theory.  A proof of the Connes-Kasparov isomorphism for real reductive groups that borrows heavily from tempered representation theory  was  announced by Wassermann in \cite{Wassermann87}. See \cite{ClareHigsonSongTang24,ClareHigsonSong24} for a full account of this, as well as for some remarks on the history of the Connes-Kasp\-arov isomorphism.  A second proof that uses  only operator $K$-theory ideas was given by Lafforgue \cite{Lafforgue02InventMath}. The general case of all almost-con\-nected Lie groups was settled by Chabert, Echterhoff and Nest \cite{ChabertEchterhoffNest03}.

If $G$ is a real reductive group with maximal compact subgroup $K$, and if we make the simplifying assumptions (for this introduction only) that $G$ is connected, and that the symmetric space $G/K$ carries a $G$-equivariant spin structure,  then the Connes-Kasparov isomorphism may be cast as an isomorphism of abelian groups 
\[
R(K) \stackrel \cong \longrightarrow K_{d}(C^*_r(G)) 
\]
from the representation ring of the maximal compact  subgroup $K$   to the  degree  $d{=}\dim(G/K)$ $K$-theory group   of the reduced   $C^*$-algebra of $G$ (the conjecture also asserts that the other $K$-theory group of $C^*_r(G)$ is zero).
The isomorphism maps the class of an irreducible unitary representation $\tau$  to the index (in a sense made precise by Kasparov \cite{Kasparov83}) of the Dirac-type operator on $G/K$ that is obtained by coupling the spinor Dirac operator to $\tau$.  

Now, when $G$ is reductive, the $K$-theory of the reduced group $C^*$-algebra decomposes as a direct sum of infinite cyclic groups, each labeled by a distinct component of the tempered dual of $G$; see \cite[Thm.~4.9]
{ClareHigsonSongTang24}. Most, but not all, components occur as labels, and the index of any indecomposable Dirac operator   is always a generator of one of the cyclic summands \cite[Thm.~8.8]{ClareHigsonSongTang24}.  So the Connes-Kasparov isomorphism sets up a map from the irreducible representations of $K$   to the set of components of the tempered dual of $G$ that is one-to-one and mostly onto. See Figure~\ref{fig:c-k} for an instance of this near-bijection. 
\begin{figure}[ht]
    \centering 
\includegraphics[width=0.35\linewidth]{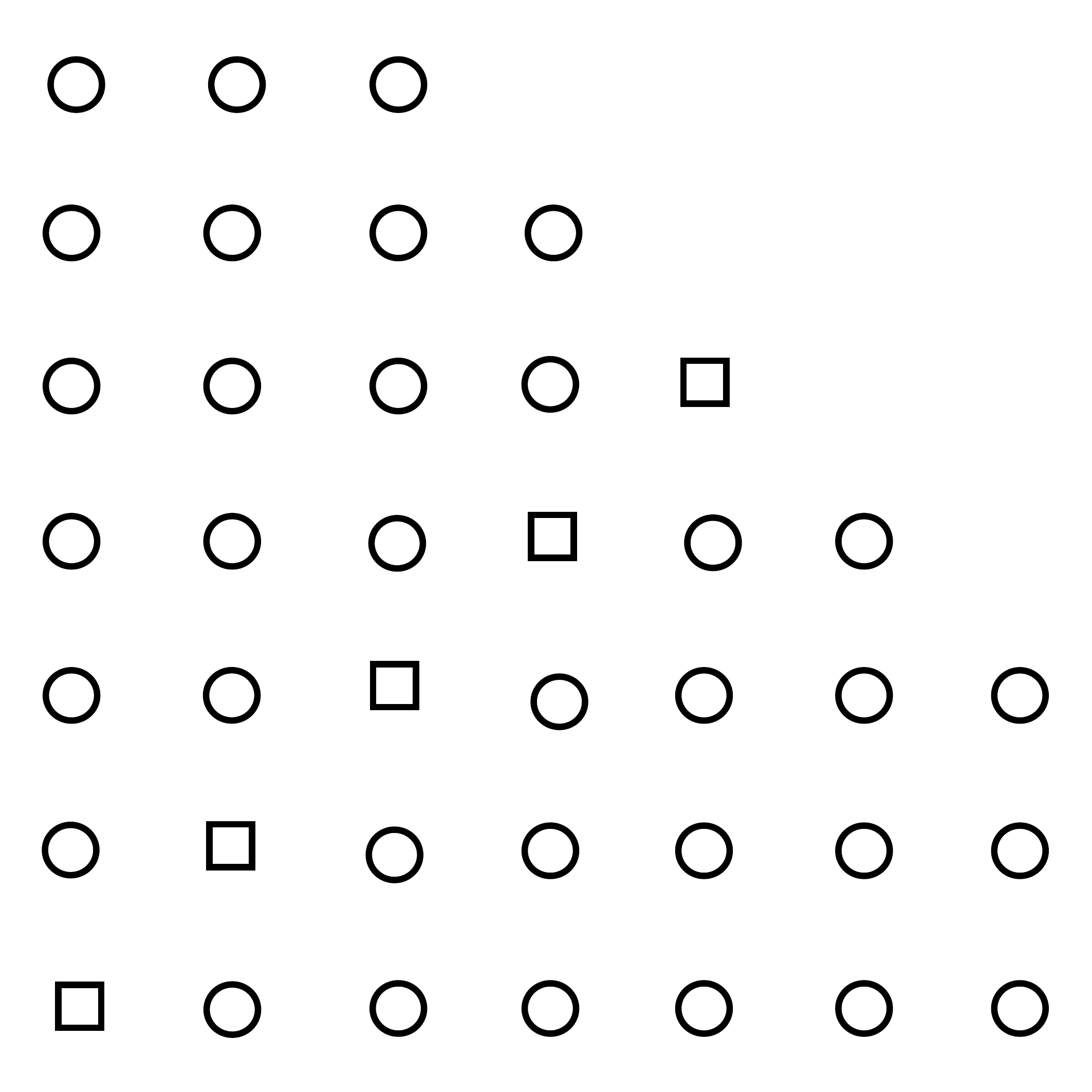}
\caption{
The Connes-Kasparov labeling of most of the components of the tempered dual of $G=Sp(1,1)$ by irreducible representations of the maximal compact subgroup  $K \cong SU(2)\times SU(2)$. The  irreducible representations of $K$ may be labeled by their highest weights, which in turn may be identified with ordered pairs of nonnegative integers. These are the nodes in the diagram.  The circles indicate that the index of the corresponding Dirac operator  is a discrete series representation; each discrete series occurs exactly once. The  squares indicate that the index of the corresponding Dirac  operator is supported on a  principal series component; the components in question are precisely those that possess two minimal $K$-types (compare Figure~\ref{fig:vogan}), and each such component occurs precisely once.   The principal series components with a single minimal $K$-type (compare Figure~\ref{fig:vogan} again) do not contribute to $K$-theory. These components are not in the range of the near-bijection mentioned in the text.
}    
\label{fig:c-k}
\end{figure}

This is strongly reminiscent of, but not the same as, Vogan's theory of minimal $K$-types for representations of a real reductive group \cite{VoganGreenBook}. According to that theory, an irreducible representation of $K$ occurs as a minimal $K$-type in representations in a unique component of the tempered dual; and in most components a single minimal $K$-type occurs. So, Vogan's theory provides a map from the  irreducible representations of $K$ to the components of the tempered dual of $G$ that is surjective and mostly one-to-one.  See Figure~\ref{fig:vogan} for an example. 
\begin{figure}[ht]
    \centering 
\includegraphics[width=0.35\linewidth]{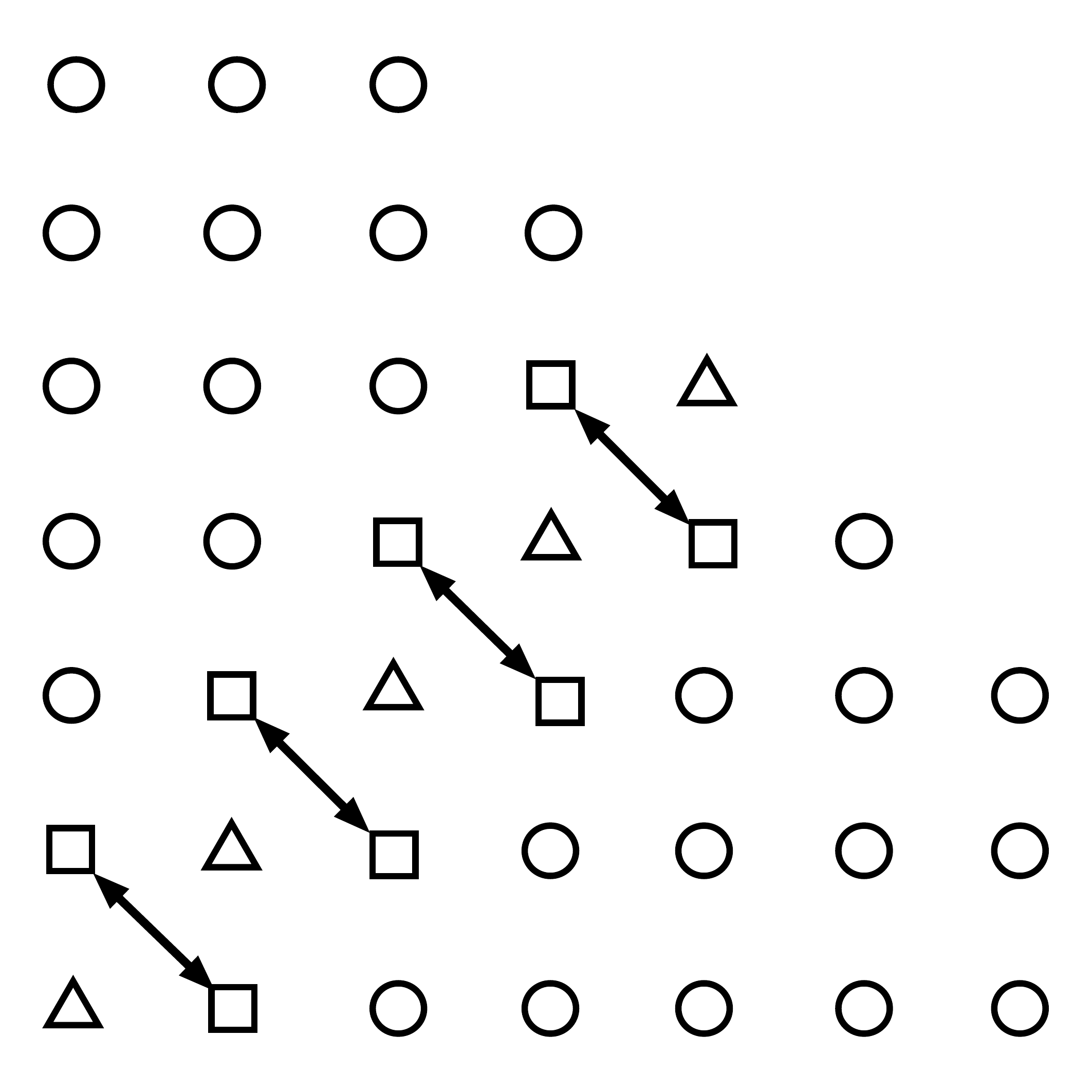}
\caption{
Vogan's labeling of the  tempered dual by minimal $K$-types in the case of $G=Sp(1,1)$, where $K = SU(2)\times SU(2)$.  The nodes in the diagram are the irreducible representations of $K$, as in Figure 1. The circles indicate that a given irreducible representation of $K$ occurs as the unique minimal $K$-type of a discrete series representation. The paired squares indicate pairs of irreducible representations of $K$ that occur as minimal $K$-types in the same principal series component of the tempered dual, while the triangles indicate irreducible representations of $K$ that occur as the unique minimal $K$-type in a principal series component of the tempered dual.  Every component of the tempered dual is listed in the diagram exactly once, except for the indicated pairings.
}    
\label{fig:vogan}
\end{figure}

The purpose of this paper is to present a new view of the Connes-Kas\-parov isomorphism that brings the two pictures of the tempered dual still closer together. 

Following a suggestion of  Vogan (see the acknowledgments below), we shall  consider a sort of enlargement of the reduced group $C^*$-algebra, although it  is  a bit more natural to formulate this enlargement using $C^*$-cat\-egories, rather than $C^*$-algebras (this is  a minor point,  and $C^*$-categories play a  limited role in what follows; see Section~\ref{subsec-category-of-psdos} for  further discussion).  

Let $G$ be an  almost-connected Lie group (meaning that $G$ is a Lie group with only  finitely many path components), and let $K$ be a maximal compact subgroup of $G$. Form   a $C^*$-category $\PSDO^*_{G,K}$ as follows: 
\begin{enumerate}[\rm (i)]

\item The objects of the category are the finite-dimensional, unitary representations of the Lie group  $K$. 

\item If $V_1$ and $V_2$ are finite-dimensional unitary representations of $K$, then the Banach space  $\PSDO^*_{G,K} (V_1,V_2)$ of morphisms from $V_1$ to $V_2$ in our $C^*$-category  is the completion, in the Hilbert space operator norm,  of the space of equivariant, properly supported,  order zero classical pseudodifferential operators 
acting between the sections of the homogeneous vector bundles over $G/K$ associated to $V_1$ and $V_2$.
\end{enumerate}
The above $C^*$-category includes the $C^*$-category generated by smoothing operators as a subcategory and an ideal, and this ideal is a $C^*$-categorical variation on the reduced group $C^*$-algebra of $G$. Now, $K$-theory groups may be defined for any $C^*$-category in a way that closely mimics the definition for $C^*$-algebras (see Section~\ref{subsec-category-of-psdos} again), and  our reworked version of the Connes-Kasparov isomorphism is as follows:

\begin{theorem*}[Theorem~\ref{thm-k-theory-of-the-c-star-category} below]
Let $G$ be an almost-connected Lie group and let $K$ be a maximal compact subgroup of $G$.
Denote by $R(K)$ the representation ring of the compact group $K$. There is an isomorphism of abelian groups 
\[
\CoKa\colon R(K) \stackrel \cong \longrightarrow K_0 (\PSDO^*_{G,K}) ,
\]
under which the class  in $R(K)$ of a finite-dimensional   unitary representation $V$ of $K$ corresponds to the class   of the identity operator on the bundle over $G/K$ with fiber $V$. In addition, $K_1(\PSDO^*_{G,K}) =0$.
\end{theorem*}

We shall prove the theorem above using the Connes-Kasparov isomorphism for $G$. Conversely, the theorem implies the Connes-Kasp\-arov isomorphism. Note   that the statement involves no Dirac operators, and indeed the isomorphism in the theorem is defined using only identity operators. As we shall see, this makes it possible to calculate  the isomorphism in simple, representation-theoretic terms.

For the rest of this introduction, we shall specialize from almost-connected Lie groups to the case of a real reductive group $G$.  Denote by $\Fin^*$ the $C^*$-category of finite-dimensional Hilbert spaces. If $\pi\colon G \to U(H_\pi)$ is an  irreducible, tempered unitary representation of $G$, then there is a functor
\[
\mult_{\pi} \colon \PSDO^*_{G,K} \longrightarrow \Fin^* 
\]
that, on objects, maps a representation $V$ of $K$ to the space 
$ [ H_\pi \otimes V]^K$.
Taking $K$-theory, the functor induces a morphism of abelian groups
\[
\mult_{\pi}\colon K_0( \PSDO^*_{G,K}) \longrightarrow  \Z.
\]
Let us now introduce the following terminology, suggested to us by Alexandre Afgoustidis. 

\begin{definition*}
A unitary representation of $G$ is \emph{tempiric} if it is tempered, irreducible, and has real infinitesimal character, as in \cite{Vogan00}, for example.  Denote by $R(G)_{\mathrm{tempiric}}$ the free abelian group on the set unitary equivalence classes of tempiric representations of $G$.
\end{definition*}

A well-known and important theorem of David Vogan  provides a bijection from irreducible unitary representations of $K$ to tempiric representations of  $G$, given by minimal $K$-type (this follows from Vogan's algebraic classification of the admissible dual in \cite{VoganGreenBook}, but for an explicit statement of the result see \cite[Thm.~1.2]{VoganBranching07}). The bijection is illustrated in Figure~\ref{fig:vogan}, where the tempiric representations are precisely the discrete series and the irreducible constituents of the base principal series representations.  

Now, define a homomorphism of abelian groups 
\[
 \mult\colon K_0(\PSDO^*_{G,K})  \longrightarrow R(G)_{\mathrm{tempiric}}
\]
by means of the formula 
\[
\mult \colon c \longmapsto \sum_{[\pi]} \mult_{\pi}(c)[\pi] ,
\]
where the sum is over representatives of the unitary equivalence classes of tempiric representations of $G$ (it  is actually a finite sum, since all but finitely many of the integer multiplicities are zero).  Using our first theorem,  we may prove as a consequence of Vogan's theorem that:

\begin{theorem*}[Compare  \eqref{eq-k-theory-and-tempirics} below]
 The  above group homomorphism
$
 \mult $
 is an isomorphism of abelian groups.
\end{theorem*}

Indeed, Vogan's theorem implies that the composition
\begin{equation}
\tag{$*$} \xymatrix{
R(K) \ar[r]^-{\CoKa}_-{\cong} &  K_0(\PSDO^*_{G,K}) \ar[r]^-{\mult} &  R(G)_{\mathrm{tempiric}}
}
\end{equation}
is an isomorphism of abelian groups; in effect, using Vogan's bijection, the matrix of the composite homomorphism is lower-triangular with all diagonal entries equal to $1$.

Our third theorem gives an explicit description of the  category $\PSDO^*_{G,K}$ in representation-theoretic terms, at least when $G$ has real rank one.  In the paper \cite{ClareCrispHigson16}, an extensive amount of information from tempered representation theory, most of it due to Harish-Chandra, was combined to produce a $C^*$-algebra isomorphism
\[
\tag{$**$}
C^*_r (G) \stackrel \cong \longrightarrow \bigoplus _{[P,\sigma]}
C_0\bigl  (\mathfrak{a}^*_P, \mathfrak{K}(H_\sigma  )\bigr ) ^{W_\sigma}
\]
that describes the reduced $C^*$-algebra of any real reductive group in rep\-resent\-ation-theoretic terms. The notation is described in more detail in Section~\ref{subsec-fourier-transform-of-c-star-algebra}, at least in the cases of interest in this paper,   but in brief, $\mathfrak{a}_P^*$ is a finite-dimensional real vector space, $H_\sigma$ is a Hilbert space, and $W_\sigma$ is a finite group that acts on the bundle $\mathfrak{a}_P^* \times H_\sigma$.  Now if $\mathfrak{v}$ is a finite-dimensional real vector space, then denote by $\overline{\mathfrak{v}}$ the compactification of $\mathfrak{v}$ that adds a sphere at infinity.

\begin{theorem*}[See Theorem~\ref{thm-fourier-isomorphism-in-real-rank-one} below] 
If $G$ has real rank one, then the   isomorphism \textup{($**$)} determines isomorphisms
\[
\PSDO^*_{G,K} (V_1, V_2) \stackrel \cong \longrightarrow \bigoplus_{[P,\sigma]} C \bigl (\overline{\mathfrak{a}^*_P}, \Compact ([V_1\otimes H_\sigma]^K,[V_2\otimes H_\sigma]^K )\bigr )^{W_\sigma} .
\]
\end{theorem*}

The theorem allow us to \emph{directly} compute the $K$-theory of $\PSDO^*_{G,K}$ in terms of tempiric representations (without reference to either the Connes-Kasparov isomorphism or Vogan's theorem). Indeed, by composing the isomor\-phisms in the theorem with evaluation at $0$ in each $\mathfrak{a}_P^*$ we obtain a homotopy equivalence from $\PSDO^*_{G,K}$ to the $C^*$-category whose objects are the finite-dimensional unitary representations of $K$ and whose Hom-sets are the finite-dimensional spaces 
\[
\bigoplus_{[P,\sigma]}  \Compact ([V_1\otimes H_\sigma]^K,[V_2\otimes H_\sigma]^K )\bigr )^{W_\sigma},
\]
and the $K$-theory of the latter may be  identified with $R(G)_{\mathrm{tempiric}}$.
This gives an independent proof of our second  theorem above, for real rank-one groups, that does not refer  to  Vogan's theorem. 

Let us remark, finally, that  for any real reductive group  $G$, the map ($*$) is readily computed to be
\[
\begin{gathered} 
R(K)\longrightarrow R(G)_{\mathrm{tempiric}}
\\
[\tau]\longmapsto \sum _{\pi\,\mathrm{tempiric}} \operatorname{mult}(\tau; \pi) [\pi] ,
\end{gathered}
\]
where $\operatorname{mult}(\tau, \pi)\in \{0,1,2,\dots\}$ is the multiplicity with which $\tau$ occurs in  $\pi$. Our first and third theorems imply that this map is an isomorphism of abelian groups (at least for real rank-one groups);  the isomorphism can be viewed as a $K$-theoretic version of Vogan's theorem about minimal $K$-types. In a separate work that developed from this paper, this last statement is proved to be a consequence of the Connes-Kasparov isomorphism for \emph{all} real reductive  groups \cite{BraddHigsonYuncken24}.
One can speculate that  by changing  the  type of pseudodifferential operators used,  the method of pseudodifferential operators developed here might be  extendable to new cases. But we shall not pursue that in this paper.

 {\sc Acknowledgements.} 
Work on this project was supported by  the NSF grant DMS-1952669, and  was carried out, in part, within the online Research Community on Representation Theory and Noncommutative Geometry (RTNCG),  sponsored by the American Institute of Mathematics.  The project was prompted by a question asked in RTNCG by David Vogan during  lectures given by the second author; see \cite{HigsonRTNCG22}. We are very grateful to Vogan  for his interest. We are also very grateful to the referee, who provided us with a great many illuminating comments and suggestions.

\section{Pseudodifferential Operators}
\label{sec-pseudodifferential-ops}

In this section we shall rapidly review most of the facts about pseudodifferential operators that we shall need in the paper.

\subsection{Pseudodifferential operators on Euclidean space} 
\label{sec-psdos-on-r-n}
For our purposes, a convenient general reference for pseudodifferential operators is Chapter 18 of the treatise \cite{HormanderVolIII} of H\"ormander.  The pseudodifferential operators of order $m\in \Z$ on $\R^n$ are precisely the   operators 
\begin{equation}
    \label{eq-an-operator-on-r-n}
A \colon C_c^\infty (\R^n) \longrightarrow C^\infty (\R^n)
\end{equation}
of the form
\begin{equation}
\label{eq-fmla-for-psdo}
(A\varphi )(x) =  (2\pi)^{-n} \int _{\R^n} e^{i\xi\cdot x} a(x,\xi) \hat \varphi (\xi)\, d\xi ,
\end{equation}
for which  the \emph{complete symbol function} $a(x,\xi)$ belongs to the \emph{symbol class} $S^m (\R^n{\times} \R^n)$ defined in  \cite[Def.\ 18.1.1]{HormanderVolIII}. Here $\hat \varphi$ is the Fourier transform of $\varphi$ \cite[Def.\,7.1.1]{HormanderVolI}, normalized so that if $a(x,\xi)\equiv 1$, then $A\varphi = \varphi$. Among other things, the conditions on the symbol class imply that the integrand is an integrable function. The operator $A$   is continuous for the usual  topologies on $C_c^\infty (\R^n)$ and $C^\infty (\R^n)$ \cite[Thm.\ 18.1.6]{HormanderVolIII}.

When the complete symbol function admits an asymptotic expansion in homogeneous functions of strictly decreasing integer order, as detailed in \cite[Prop.\ 18.1.3 and Def.\ 18.1.5]{HormanderVolIII}, the operator $A$ will be said to be \emph{classical}.  We shall be exclusively concerned with classical operators in what follows.

Every pseudodifferential operator $A$  on $\R^n$  is \emph{adjointable}, in the sense that there is a linear operator 
\[
A^* \colon C_c^\infty (\R^n) \longrightarrow C^\infty (\R^n) 
\]
such that 
\[
\langle A\varphi,\psi \rangle_{L^2} = \langle \varphi, A^* \psi\rangle_{L^2} \qquad \forall \varphi,\psi\in C_c^\infty (\R^n) .
\]
If $A$ is a classical  pseudodifferential operator, then the adjoint is a classical pseudodifferential operator, too, and it has the same order as $A$. See \cite[Thm.~18.1.7]{HormanderVolIII}.  

The composition of properly supported,\footnote{For the definition of \emph{properly supported}, see \cite[Def.\ 18.1.21]{HormanderVolIII} or Section~\ref{subsec-properly-supported} below.}  pseudodifferential operators is again pseudodifferential (and classical if the factors are classical), and the order of the composite operator is the sum of the orders of the factors \cite[Thm.~18.1.8]{HormanderVolIII}.

\subsection{Pseudodifferential operators on smooth manifolds} 
\label{sec-psdos-on-manifolds}
The space of all classical pseudodifferential operators is  invariant under  diffeomorphisms \cite[Thm.~18.1.17]{HormanderVolIII}. Because of this, one can transport the definitions of  pseudodifferential and classical pseudodifferential operators to any smooth manifold, as in  \cite[Def.~18.1.20]{HormanderVolIII}.  

H\"ormander has given  a coordinate-free definition of classical pseudodifferential operators that is conceptually  helpful \cite{Hormander65}.  If $M$ is a smooth manifold without boundary, then a continuous linear operator
\[
A \colon C_c^\infty (M) \longrightarrow C^\infty (M)
\]
is a classical pseudodifferential operator of order $m$ (according to the  coor\-dinate-free definition of  H\"ormander) if and only if for every $\varphi \in C_c^\infty (M) $ and every bounded set of functions $\ell \in C^\infty (M)$ with $d\ell$ nowhere zero throughout $\Supp(\varphi)$, there is a uniform  asymptotic expansion
\begin{equation}
    \label{eq-coordinate-free-asymptotic-expansion}
e^{-i t \ell } \cdot A(e^{i t\ell }\varphi )
\sim 
\sum_{k=0}^\infty \alpha_k(\varphi ,\ell ) t ^{m-k}\qquad\text{as}\quad  t \to +\infty,
\end{equation}
with $\alpha_k(\varphi ,\ell )\in C^\infty (M)$, meaning that for every $N$,  the set of all  functions 
\begin{equation}
    \label{eq-coordinate-free-asymptotic-expansion-uniformity-condition}
t ^{N-m}\cdot \bigl (
e^{-i t \ell } \cdot A(e^{i t\ell }\varphi )
-  \sum_{k=0}^{N-1} \alpha_k(\varphi ,\ell ) t ^{m-k}\bigr ) \in C^\infty(M),
\end{equation}
for all $\ell $ and all $t \ge 1$, is a bounded set in $C^\infty (M)$.  See \cite[Sec.\ 2]{Hormander65} for further details and discussion.

H\"ormander proves that when $M {=} \R^n$, the compactly supported operators of the above type, meaning those for which there exists $\varphi  \in C_c^\infty(M)$ with 
\begin{equation}
    \label{eq-cutoff-function-phi}
A \varphi\psi   = A\psi = \varphi  A\psi\qquad \forall \psi \in C^\infty _c(M)    ,
\end{equation}
are precisely the compactly supported pseudodifferential operators of the type considered in the previous section. The complete symbol function is characterized by 
\[
a(x,\xi) = \bigl ( e^{-i \xi } \cdot A( e^{i\xi}\varphi )\bigr ) (x) 
\]
where $\xi\in \R^n$ is regarded as a linear function on $\R^n$ using the standard dot product on $\R^n$, and $\varphi$ is as in \eqref{eq-cutoff-function-phi} above.

It follows that the pseudodifferential operators in H\"ormander's coord\-inate-free definition are precisely the operators with the following property: for every $\varphi  \in C^\infty _c(M)$ and every diffeomorphism from a neighborhood of the support of $\varphi $ to an open subset of $\R^n$, the operator $\varphi  A \varphi $ transfers, via that diffeomorphism, to a (compactly supported) pseudodifferential operator on $\R^n$ in the sense of the previous section.  These are the operators that we shall be considering from this point onward; they coincide with the operators on manifolds  defined in \cite[Def.\ 18.1.20]{HormanderVolIII}.

We shall work in  slightly greater generality.  If $M$ is a smooth manifold,  and if $E$ and $F$ are smooth vector bundles over $M$, then using local coordinates on $M$ and local frames for $E$ and $F$, one can define pseudodifferential and classical pseudodifferential operators
\begin{equation}
    \label{eq-psdo-on-sections-of-bundles}
A \colon C_c^\infty (M;E)\longrightarrow C^\infty (M;F)
\end{equation}
that act between spaces of smooth sections of $E$ and $F$; this is done in   \cite[Def.~18.1.32]{HormanderVolIII}.  The operators  are characterized by the property that for any local frames of $E_1$ and $E_2$ over an open set $U\subseteq M$,   the entries of the matrix of operators determined  by $\varphi A \varphi $ and   the frames are classical pseudodifferential operators on $U$.

\subsection{The principal symbol}

Let $A$ be a pseudodifferential operator of order $m$ on a smooth manifold $M$. The value of the leading coefficient function 
\[
\alpha_0(\varphi ,\ell)\colon M \longrightarrow \C
\] 
in the asymptotic expansion \eqref{eq-coordinate-free-asymptotic-expansion} at a point $x\in M$ depends only on the value of  $\varphi  $ at $x$  and the value  of the differential $d\ell$ at $x$.  So from $\alpha_0$ we obtain a function 
\[
a_0 \colon T^*M\setminus M  \longrightarrow \C ,
\]
that is characterized by 
\[
a_0(x,\xi)\varphi (x) = \alpha_0(\varphi ,\ell)(x)\quad\text{if \ $(d\ell)(x) = \xi $.}
\]
This is the \emph{principal symbol} of $A$. It is a smooth  function on $T^*M\setminus M$ that is  homogeneous of degree $m$ on each cotangent fiber. 
 See \cite[Lemmas 2.2 \& 2.3]{Hormander65} or \cite[pp.~82-83]{HormanderVolIII}.  
 
It follows from H\"ormander's coordinate-free definition of classical pseudodifferential operators that the principal symbol may be characterized in terms of the operator $A$ by the formula  
\[
a_0(x,\xi)\varphi (x)  = \lim_{t\to +\infty} t^{-m} e^{-i\ell(x)} (A e^{it \ell}\varphi)  (x)
\quad \text{when $(d\ell)(x)=\xi$,}
\]
and that the limit is uniform over bounded families of smooth functions $\ell$.

In the case of a pseudodifferential operator 
\[
A \colon C_c^\infty (M;E) \longrightarrow C^\infty (M;F) ,
\]
acting on bundle sections, the principal symbol of $A$ is a smooth section  of the pullback of the bundle $\Hom (E,F)$ to $T^*M\setminus M$ that is homogeneous of order $m$ in each fiber \cite[p.~92]{HormanderVolIII}.  It is characterized by the formula 
\begin{equation}
    \label{eq-limit-formula-for-principal-symbol}
a_0(x,\xi)s(x) = \lim_{t\to +\infty} t^{-m}e^{-i\ell(x)}( A  e^{it \ell}s)(x)
\quad \text{when $(d\ell)(x)=\xi$,}
\end{equation}
for all smooth, compactly supported sections $s$. Once again the limit is 
uniform over bounded families of smooth functions $\ell$.

For various purposes below, it  will be convenient to use the following notion of compactification of a   finite-dimensional real vector space, which adds to the vector space a ``sphere at infinity.''

\begin{definition}
\label{def-spherical-compactification}
  If $\mathfrak{v}$ is any finite-dimensional  vector  space, then we shall denote by $\overline{\mathfrak{v}}$ the topological compactification of $\mathfrak{v}$ (meaning a compact Hausdorff space including $\mathfrak{v}$ as a dense open subset) for which the continuous map 
  \[
 \mathfrak{v} \rightarrow  \mathfrak{v} , \qquad v \mapsto (1{+}\|v\|) ^{-1}  v
  \]
  extends to a homeomorphism from $\overline {\mathfrak{v}} $ to the closed unit ball in $\mathfrak{v}$.  Here $\|\,\cdot \, \|$ is any norm on $\mathfrak{v}$ (the compactification is independent of the choice of norm).
\end{definition}

The action of $GL(\mathfrak{v})$ on $\mathfrak{v}$ extends to a continuous action on the compactification above.  With this, if $\mathfrak{V}$ is any real vector bundle, then we may define a space $\overline{\mathfrak{V}}$ by applying the compactification construction fiberwise (compare \cite[Sec.\ 1.2]{Atiyah67}). We shall write the added boundary as $\partial \mathfrak{V}$ below.

The  principal symbol  of an order zero pseudodifferential operator on a smooth manifold $M$, being homogeneous of order zero on each fiber of $T^*M\setminus M$, extends to a continuous function (or section)  on $\overline{T^*M}   \setminus M$, and is determined by the restriction of this extension to $\partial T^*M$:
\[
a_0 \colon  \partial T^*M  \longrightarrow \C .
\]
If the fiberwise compactification is constructed using a Riemannian metric on $M$, then $\partial T^*M$ identifies with the unit sphere bundle in $T^*M$ and therefore carries a canonical smooth structure, independent of the choice of metric. 
The principal symbol is then a smooth function. Moreover, it is evident from the definitions that if the principal symbol of a  classical pseudodifferential operator of order $0$ is identically zero, then $ A$ is in fact of order ${-}1$. We therefore obtain a short exact sequence of algebras and algebra morphisms
 \[
0 \longrightarrow \Psi^{-1}_c (M) \longrightarrow  \Psi^0_c(M) \longrightarrow 
C^\infty _c(\partial T^*M)
\longrightarrow 0 ,
\]
where $\Psi_c^{-1}(M)$ and $\Psi^0_c(M)$ are the algebras of compactly supported order $-1$ and order $0$ classical pseudodifferential operators, respectively.

\subsection{The pseudodifferential extension of C*-algebras}

It will be important for our purposes that compactly supported   classical pseudodifferential operators of order $0$ extend to  bounded Hilbert space operators: 

\begin{theorem}[See for instance {\cite[Thm.\ 18.1.1]{HormanderVolIII}}]
\label{thm-l-2-boundedness-of-order-zero-psdo}
If $A$ is a compactly supported   classical pseudodifferential operator of order $0$ on a smooth manifold $M$, then $A$ is $L^2$-bounded:  for any choice of smooth measure on $M$, there is a constant $C\ge 0$ such that 
\[
\| A \varphi  \|_{L^2(M)} \le  C \| \varphi \|_{L^2(M)} \qquad \forall \varphi \in C_c^\infty (M).
\]
\end{theorem}

\begin{remark} 
H\"ormander deals with the case of $M{=}\R^n$ in the reference cited, but the result is  local in nature, so this suffices.
\end{remark}

\begin{theorem}
\label{thm-negative-order-operators-compact}
  If $A$ is a compactly supported   classical pseudodifferential operator of order $-1$, then $A$ extends to  a compact Hilbert space operator on $L^2 (M)$.
\end{theorem}

\begin{proof} 
This is again a local result, and we may assume that $M=\R^n$.  In \cite[Thm.\ 18.1.1]{HormanderVolIII} is is shown that in this case, if $A$ is compactly supported and has order $-1$, then $A$ in fact maps $L^2 (\R^n)$ into the Sobolev space $H^1 (\R^n)$.
The proof follows from this and the Rellich lemma in Sobolev space theory (see for instance \cite[Thm.\ 5.2.8]{Zimmer90}). 
\end{proof} 

We shall also require the following slightly more specialized results, related to $C^*$-algebra theory. 

\begin{theorem}[{\cite[Lem.\
 7.2]{Seeley65}}; see also {\cite[Thm.\ 5]{KohnNirenberg65}}] 
\label{thm-symbol-map-is-continuous}
If $A$ is a compactly supported,  classical pseudodifferential operator of order zero on a smooth manifold $M$, then 
\[
\| A\|_{L^2(M)\to L^2 (M)} \ge \sup \bigl \{\, |a_0(m,\xi)| : m\in M,\,\, \xi\in T^*_mM \setminus \{ 0\}\,\bigr \}
\]
\textup{(}the $L^2$-operator norm of $A$ is of course the best possible constant in Theorem~\textup{\ref{thm-l-2-boundedness-of-order-zero-psdo}}; the supremum above is finite because $a_0$ is homogeneous in $\xi$ of order $0$\textup{)}.    
\end{theorem}

It follows from Theorem~\ref{thm-symbol-map-is-continuous} that the principal symbol gives rise to a continuous homomorphism of Banach algebras, and in fact $C^*$-algebras\footnote{We have already noted that adjoint of a compactly supported, classical pseudodifferential operator $A$ of order zero is another compactly supported, classical pseudodifferential operator of order zero. The adjoint operation on $C_0(\partial T^*M)$ is the pointwise complex conjugate, and the principal symbol of $A^*$ is the adjoint of  the principal symbol of $A$ \cite[(18.1.11)]{HormanderVolIII}.}
\[
\overline {\Psi^0_c(M)} \longrightarrow C_0(\partial T^*M),
\]
where the overbar denotes the $C^*$-algebra completion in the operator norm.

The following result, or rather the equivariant version   that we shall derive from it in Section~\ref{subsec-psdo-extension-in-equivariant-context}, will play a crucial role in the paper.

\begin{theorem}[{\cite[Thm.\ 14.1]{Seeley65}};  see also {\cite[Cor.\ 6.1]{KohnNirenberg65}}]
The above extended principal symbol map fits into a short exact sequence of $C^*$-algebras
\[
0 \longrightarrow \mathfrak{K}(L^2 (M)) \longrightarrow \overline{\Psi^0_c(M)} \longrightarrow 
C_0(\partial T^*M)
\longrightarrow 0 .
\]
\end{theorem}

\section{Operators on Proper Homogeneous Spaces}

We shall begin this section by considering    any smooth manifold  $M$ that is equipped with a smooth and  proper   \emph{right} action of a Lie group $G$.    Later in the section we shall specialize to the case where $M=K \backslash G $ and  $K$ is a  compact subgroup of $G$. 

\subsection{Sections of equivariant vector bundles}
Let $E$   be a  smooth, right-$G$-equi\-variant, hermitian vector bundle over $M$, and denote by $C^\infty (M;E)$ and $C_c^\infty (M;E)$ its spaces of smooth  sections, and smooth, compactly supported sections, respectively.  Both spaces carry right actions of the group $G$ that are defined by 
\[
(s \cdot g) (m) =   s(mg^{-1})\cdot g \qquad \forall g \in G,\,\,\, \forall m \in M .
\]
The right-hand side of the formula involves the  action 
\[
E_{mg^{-1}} \to E_m, \quad e\mapsto e \cdot g
\]
of the group element $g$ between the fibers of $E$.

Now fix a $G$-invariant smooth measure on $M$ and a {left}-invariant Haar measure on $G$.   Denote by $C_c^\infty (G)$ the space of smooth, compactly supported complex functions on $G$, equipped with the standard   convolution multiplication and the involution that is used in groupoid $C^*$-algebra theory \cite[Sec.\,II.1]{RenaultGroupoidApproach80}:
\begin{equation}
\label{eq-convolution-for-left-haar-measure}
(f_1*f_2)(g) = \int _G f_1(gg_1)f_2(g_1^{-1} ) \, dg_1\quad\text{and}\quad 
f^*(g) = \overline{f(g^{-1})}  .
\end{equation}
The spaces  $C^\infty (M;E)$ and $C_c^\infty (M;E)$ carry right module actions of the algebra $C_c^\infty (G)$ defined by the formula 
\begin{equation}
    \label{eq-right-algebra-action}
(s* f) (m) = \int _G  ( sg_1^{-1}) (m) f(g_1^{-1}) \, dg_1 \quad 
\forall s \in C^\infty (M;E)\,\,\,\forall f \in C_c^\infty (G) .
\end{equation}
Here the scalar $f(g)$ has been written on the right-hand side of the vector $(sg)(m)$ in an effort to make the formula easier to parse.  When $M{=}G$ and when $E$ is trivial, so that $C_c^\infty(M,E) = C_c^\infty(G)$, \eqref{eq-right-algebra-action} agrees with \eqref{eq-convolution-for-left-haar-measure}.

The space $C^\infty (M;E)$  also carries the obvious $L^2$-inner product, and indeed a $C_c^\infty (G)$-valued inner product
\begin{equation}
    \label{eq-inner-product}
\langle\,\,\,\,,\,\,\,\rangle_{C^\infty _c (G)} : C_c^\infty (M;E) \times C_c^\infty (M;E) \longrightarrow C_c^\infty (G)
\end{equation}
that is defined by 
\[
\langle s_1, s_2 \rangle _{C^\infty _c (G)} : g \longmapsto    \langle s_1,  s_2 g^{-1} \rangle_{L^2(M;E)} .
\]
When $C_c^\infty (M,E) = C_c^\infty (G)$, this is simply $s_1^**s_2\in C_c^\infty (G)$, where the involution and product are as in \eqref{eq-convolution-for-left-haar-measure}.  In general, the inner product is complex-sesquilinear and satisfies 
\[
\begin{aligned} 
\langle s_1, s_2* f \rangle _{C^\infty _c (G)} & = \langle s_1, s_2 \rangle _{C^\infty _c (G)}* f\\
\langle s_1, s_2 \rangle _{C^\infty _c (G)} & = \langle s_2, s_1 \rangle _{C^\infty _c (G)}^*
\end{aligned}
\]
for all $s_1,s_2\in C_c^\infty (M;E)$ and all $f\in C_c^\infty (G)$.

We shall use the realization of the reduced $C^*$-algebra of $G$ from groupoid $C^*$-algebra theory; see \cite[Sec.\,II.1]{RenaultGroupoidApproach80} again. The $*$-algebra $C^\infty _c (G)$ is represented faithfully by left-convolution (using the left Haar measure) as bounded operators on the Hilbert space $L^2 (G,dg^{-1})$ (this is a $*$-represen\-tation), and $C^*_r(G)$ is by definition the completion in the operator norm. When $G$ is unimodular, which is our main interest, this is the usual definition of the reduced $C^*$-algebra \cite[Sec.\,7.2]{PedersenCstarBook}. In general, pointwise multiplication by $\Delta(g)^{-1/2}$, where $\Delta$ is the modular function of $G$, as in \cite[Sec.\,7.1]{PedersenCstarBook}, is an algebra automorphism of $C_c^\infty(G)$ that extends to a $*$-isomorphism from  $C^*_r(G)$ as we have defined it to the standard version of the reduced $C^*$-algebra.

The inner product \eqref{eq-inner-product} has the additional properties that 
\[
\langle s,s\rangle _{C^\infty_c (G)} \ge 0 \quad \text{and} \quad 
\langle s,s\rangle _{C^\infty_c (G)} = 0 \,\, \Rightarrow \,\, s =0 ,
\]
where the inequality indicates positivity in the $C^*$-algebra $C^*_r (G)$. The formula 
\[
\| s\|^2   = \| \langle s,s\rangle_{C^\infty _c (G)} \|_{C^*_r (G)} 
\]
defines a norm on the complex vector space $C_c^\infty (M;E)$.  We shall denote by $C^*_r(M;E)$ the Banach-space completion of $C^\infty _c (M;E)$ in this norm.  The right action in \eqref{eq-right-algebra-action} and inner product in \eqref{eq-inner-product} extend to completions to give $C^*_r (M;E)$ the structure of a \emph{Hilbert $C^*$-module} over $C^*_r (G)$. See \cite{Lance95} for background information on Hilbert $C^*$-modules. The above constructions are adapted from   \cite[Sec.~3]{BaumConnesHigson93}, but they are  due to Kasparov \cite{Kasparov88}.

For the rest of this section we shall specialize to the situation where    $K$ is a compact subgroup of $G$ and   $M = K \backslash G$ (the space of right $K$-cosets).  Each   finite-dimensional unitary representation $V$ of $K$ determines an \emph{induced}   $G$-equivariant   Hermitian vector bundle  
\[
  K \backslash ( V \times G) \to K \backslash G.
\]
Its fiber  over $Kg\in K\backslash G$ is the set of all equivalence classes $[v,h]$ of elements in $ V\times Kg$ under the equivalence relation 
\[
(v,h)\sim (kv,kh).
\]
If $C_c^\infty (G)$ is equipped with the usual left-translation action of $K$ (arranged to be a left action), then     there is a unique  isomorphism of $G$-vector spaces 
\begin{equation}
    \label{eq-basic-uncompleted-module-isomorphism}
\bigl [  V \otimes C_c^\infty (G) \bigr ]^K \stackrel \cong \longrightarrow C_c^\infty \bigl (K \backslash G; K \backslash ( V {\times} G)
\bigr)  
\end{equation}
 such that 
\[
 \sum_j  v_j \otimes f_j   \longmapsto \Bigl [ Kg \mapsto  \sum_j   f_j(g) [ v_j,g]\Bigr].
\]
We shall mostly use this view of the sections of the induced bundle from now on. 

Under the isomorphism \eqref{eq-basic-uncompleted-module-isomorphism}, if $E$ is the Hermitian vector bundle over $M$ induced from $V$, then the right action of $C_c^\infty (G)$ on $ C_c^\infty (M;E)  $ given by \eqref{eq-right-algebra-action} corresponds to the right action on the left-hand side in \eqref{eq-basic-uncompleted-module-isomorphism}  given by the simple formula
\[
\bigl (  \sum_j v_j \otimes f_j    \bigr)  * f =  \sum_j v_j \otimes (f_j* f)     ,
\]
while the inner product \eqref{eq-inner-product} corresponds to   
\[
\bigl \langle   \sum_j  v_j\otimes f_j   ,  
 \sum_k w_k\otimes h_k      \bigr \rangle_{C_c^\infty (G)} =
 \sum_{j,k}  \langle v_j,w_k\rangle_{V} \ ( f_j{}^* * h_k ) .
\]
The isomorphism \eqref{eq-basic-uncompleted-module-isomorphism}  extends to a unitary isomorphism 
\begin{equation}
    \label{eq-basic-hilbert-module-isomorphism}
 \bigl [ C^*_r (G) \otimes V\bigr]^K \stackrel \cong\longrightarrow  C^*_r \bigl (K \backslash G, K \backslash ( V {\times} G) \bigr )
\end{equation}
between right Hilbert $C^*$-modules over $C^*_r (G)$.

\subsection{Properly supported operators}
\label{subsec-properly-supported}
Let $M$ be   a general smooth manifold (for a moment) and let  $E_1$ and $E_2$ be  smooth   vector bundles over $M$.    A complex-linear operator
\[
T \colon C_c^\infty (M;E_1) \longrightarrow C^\infty (M;E_2)
\]
is \emph{properly supported} if for every   $\varphi\in C_c^\infty (M)$ there exists some $\psi \in C_c^\infty(M)$ such that 
\[
(1-\psi) T \varphi = 0 \quad\text{and} \quad \varphi T (1 - \psi) = 0.
\]
In this case $T$ actually maps smooth, compactly supported sections of $E_1$ to smooth, compactly supported sections of $E_2$.  Moreover there is a unique extension of $T$  to a linear operator 
\[
T \colon C^\infty (M;E_1) \longrightarrow C^\infty (M;E_2)
\]
with the property that if $\varphi$ and $\psi$ are smooth compactly supported functions, and if $\varphi T (1 {-} \psi) = 0$, then 
\[
\varphi T s = \varphi T \psi s \quad \forall s \in C^\infty (M;E_1).
\]

Now assume that $E_1$ and $E_2$ are equipped with Hermitian structures, and that $M$ is equipped with a smooth measure. The operator $T$ is  \emph{adjointable} if there exists  an operator 
\[
T^* \colon C_c^\infty (M;E_2) \longrightarrow C^\infty (M;E_1)
\]
such that 
\[
\langle Ts_1,s_2\rangle _{L^2(M;E_2)} = \langle s_1,T^*s_2\rangle _{L^2(M;E_1)} 
\]
for all $s_1\in C_c^\infty (M;E_1)$ and $s_2\in C_c^\infty (M;E_2)$.  The property of being  adjointable is independent of the choice of measure (but the adjoint operator depends on the choice).

\begin{proposition} 
Let $K$ be a compact subgroup of $G$ and let $V_1$ and $V_2 $ be  finite-dimensional unitary representations of $K$.   
If $T\colon  [C_c^\infty (G)\otimes V_1]^K \to [C_c^\infty (G)\otimes V_2]^K $ is properly supported, adjointable and equivariant, then $T$ extends to a bounded operator 
\[
T\colon  [L^2 (G)\otimes V_1]^K \to [L^2 (G)\otimes V_2]^K,
\]
if and only if it  extends to a bounded and adjointable operator between Hilbert $C^*$-modules over $C^*_r (G)$,
\[
T\colon  [C^*_r (G)\otimes V_1]^K \to [C^*_r (G)\otimes V_2]^K.
\]
In this case, the two extensions have the same norm.
\end{proposition} 

\begin{proof} 
To simplify the notation let us consider only the case where $K =\{ e\}$ and the representations $V_1$ and $V_2$ are (necessarily) trivial, and of dimension one.  Apart from issues of notation, the general case is no different.

If $f\in C_c^\infty (G)$, then 
\[
\| Tf \|_{C^*_r (G)} = 
\sup \, \bigl \{ \, \| (Tf)h\| _{L^2(G)} : h\in C_c^\infty (G),\,\, \| h\|_{L^2 (G)}=1\,\bigr \} .
\]
But 
\begin{multline*}
\| (Tf)h\| _{L^2(G)} = \| T(fh)\| _{L^2(G)} 
\\
 \le \|T\|_{L^2 \to L^2} \|fh\|_{L^2(G)} \le 
\|T\|_{L^2 \to L^2} \|f\|_{C^*_r (G)}\|h \|_{L^2(G)},
\end{multline*}
which shows that  $\|T\|_{C^*_r \to C^*_r }\le\|T\|_{L^2 \to L^2} $.

In the opposite direction let $\{f_n\}$ be an approximate unit for $C^*_r(G)$ that consists of elements in $C^\infty _c(G)$.
If $k\in C_c^\infty(G)$, then $f_n k \to k$ in the $L^2$-norm, and so for any $h\in C_c^\infty (G)$, 
\begin{multline*}
\| Th \|_{L^2(G)}^2 =
\bigl \langle Th , Th \bigr \rangle_{L^2 (G)}  
=
\bigl \langle  h , T^*Th \bigr \rangle_{L^2 (G)}  
= 
\lim_{n\to \infty}\bigl \langle  f_n h , T^*Th \bigr \rangle_{L^2 (G)} 
\\
=
\lim_{n\to \infty}\bigl \langle  T^*T (f_n h) , h \bigr \rangle_{L^2 (G)}  
=
\lim_{n\to \infty}\bigl \langle  (T^*T f_n) h , h \bigr \rangle_{L^2 (G)}  .
\end{multline*}
So the Cauchy-Schwarz inequality gives 
\begin{multline*}
\| Th \|_{L^2(G)}^2 
\le \limsup \| (T^*T f_n) h\|  _{L^2(G)} \cdot \|  h\|  _{L^2(G)}
\\
\le \limsup \| (T^*T f_n)\|  _{C^*_r(G)}\cdot \|  h\|^2 _{L^2(G)}  
\le \|T\|^2 _{C^*_r \to C^*_r} \| h\|^2 _{L^2 (G)}.
\end{multline*}
This shows that $\|T\|_{L^2 \to L^2} \le \|T\|_{C^*_r \to C^*_r } $.
\end{proof} 

\subsection{Smoothing operators}
\label{subsec-smoothing-operators}
Let $M$ be a smooth manifold. If $k$ is a smooth section of the bundle $\Hom (E_1,  E_2)$ over $M\times M$ (where $E_2$ is pulled back to the product $M\times M$ along the left coordinate projection, and $E_1$ along the right), then the formula 
\[
(T s)(x) = \int _M   k(x,y)\cdot s(y) \, dy ,
\]
where the dot represents the contraction operation from $\Hom (E_1,E_2)\otimes E_1$ to $E_2$, defines a \emph{smoothing operator}
\[
T \colon C_c^\infty (M;E_1) \longrightarrow C^\infty (M;E_2)
\]
acting between sections of the bundles $E_1$ and $E_2$.  The operator $T$ is properly supported if and only if $k$ is properly supported, in which case $T$ is also automatically adjointable.  It is equivariant if and only if $k$ is a $G$-equivariant section for the diagonal action of $G$ on $M\times M$.

The purpose of this section is to describe the smoothing operators in the case where $M = K \backslash G$, and where $E_1$ and $E_2$ are the equivariant vector bundles induced from finite-dimensional unitary representations $V_1$ and $V_2$ of the compact group $K$.

The space of all smooth integral kernels may be described in this case as follows. Let    $K{\times} K$ act by left translation on $G{\times }G $ in the obvious way, and on the vector space  $\Hom (V_1,V_2)$ by $(k_1,k_2) \cdot T = k_1 Tk_2^{-1}$. There is a   unique isomorphism of $G$-vector spaces 
\begin{equation}
\label{eq-computational-form-of-integral-kernels}
 [  C^\infty (G{\times} G)\otimes \Hom (V_1,V_2) ]^{K\times K} \stackrel \cong \longrightarrow C^\infty (M{\times} M ; \Hom (E_1,E_2 ))
\end{equation}
for the diagonal action of $G$ such that 
\begin{equation}
\label{eq-fmla-for-integral-kernels}
\sum_j  F_j \otimes T_j   \mapsto h, 
\qquad h (Kg_1,Kg_2)\colon [v,g_1]\mapsto    \sum_j F_j(g_1,g_2)  [T_jv,g_1].
\end{equation} 
The action of $G$ on   \eqref{eq-computational-form-of-integral-kernels} is through the diagonal right-translation action on $G\times G$ and on $M\times M$. The map in \eqref{eq-computational-form-of-integral-kernels} determines an isomorphism 
\begin{equation}
\label{eq-computational-form-of-equivariant-integral-kernels}
 [  C^\infty (G )\otimes \Hom (V_1,V_2) ]^{K\times K} \stackrel \cong \longrightarrow C^\infty (M{\times} M ; \Hom (E_1,E_2 ))^G
\end{equation}
such that 
\begin{equation}
\label{eq-fmla-for-equivariant-integral-kernels}
\sum_j  f_j \otimes T_j   \mapsto h ; 
\qquad h (Kg_1,Kg_2)\colon [v,g_1]\mapsto    \sum_j f_j(g_2g_1^{-1})  [T_jv,g_1].
\end{equation} 
The  image of the restriction of \eqref{eq-computational-form-of-equivariant-integral-kernels} to $C_c^\infty (G)$ is the space of properly supported, $G$-equivariant smooth integral kernels.

Under the isomorphisms \eqref{eq-basic-uncompleted-module-isomorphism} and \eqref{eq-fmla-for-equivariant-integral-kernels}, the action of equivariant smoothing operators on the smooth compactly supported sections of $E_1$ is given by the natural map 
\begin{equation}
\label{eq-action-of-smoothing-ops-in-standard-form}
 [  C^\infty (G)\otimes \Hom (V_1,V_2) ]^{K\times K}  \times 
 [ C^\infty_c (G) \otimes V_1]^{K }
 \longrightarrow 
 [   C^\infty (G)\otimes V_2 ]^{K }
\end{equation}
characterized by 
\begin{equation}
\label{eq-fmla-for-action-of-smoothing-ops-in-standard-form}
\Bigl ( \sum_j   f_j \otimes T_j  ,
\sum _k  h_k\otimes v_k\Bigr ) \longmapsto 
\sum_{j,k}   (f_j {*} h_k) \otimes T_j v_k .
\end{equation}
This is under the assumption that the invariant measure on $M = K \backslash G$ is chosen so that the integral of any function on $M$ is equal to the Haar integral of its pullback to $G$. 

Of course, properly supported smoothing operators may composed, and their adjoints taken.  We omit the natural formulas for these operations under  the isomorphism \eqref{eq-fmla-for-equivariant-integral-kernels}.  It is clear from the explicit formula \eqref{eq-fmla-for-action-of-smoothing-ops-in-standard-form} that:

\begin{lemma} 
Every $G$-equivariant, properly supported smoothing operator 
\[
T \colon  [ C_c^\infty  (G)\otimes V_1]^K \longrightarrow   [ C_c^\infty  (G)\otimes V_2]^K 
\]
extends to bounded   operator   
$ [ L^2 (G)\otimes V_1]^K \to  [L^2  (G)\otimes V_2]^K$,
and to a bounded and adjointable operator 
 $[C_r^*  (G)\otimes V_1]^K\to [ C_r^*  (G)\otimes V_2]^K$.  \qed
\end{lemma}
 
The operator-norm completion of the  properly supported, $G$-equivariant smoothing operators is isomorphic to the space 
\begin{equation}
\label{eq-norm-completion-f-smoothing-ops}
\bigl [ C^*_r (G)\otimes  \Hom (V_1,V_2)\bigr ]^{K\times K}
\subseteq \mathfrak{B} \bigl (  [  C^*_r(G)\otimes V_1 ]^K, [  C^*_r(G) \otimes V_2 ]^K\bigr )
\end{equation}
under the identifications  \eqref{eq-basic-uncompleted-module-isomorphism} and 
\eqref{eq-computational-form-of-equivariant-integral-kernels}-\eqref{eq-fmla-for-action-of-smoothing-ops-in-standard-form}.

\subsection{Pseudodifferential operators on proper homogeneous spaces}
\label{sec-psdos-on-proper-homogeneous-spaces}
We are interested in properly supported, $G$-equivariant, classical pseudodifferential operators on the homogeneous space $K\backslash G$. 
We shall use the following notation:

\begin{definition} Let $V_1$ and $V_2$ be finite-dimensional unitary representations of $K$ and let $m\in \Z$.
We shall denote by 
$\Psi^{m}_{G,K}(V_1,V_2)$
the space of $G$-equivariant, properly supported, classical pseudodifferential operators 
\[
A \colon [ C_c^\infty  (G)\otimes V_1]^K  \longrightarrow 
[ C_c^\infty  (G)\otimes V_2]^K 
\]
of order $m$.
\end{definition} 

Let us consider the case $m{=}0$. As we have seen,  the principal symbol of an operator in $\Psi^0_{G,K}(V_1,V_2)$ can be viewed as a smooth section over $\partial T^* M$, where $M=K \backslash G$, of the $G$-equivariant vector bundle whose fiber over each point of $\partial T^*_e M$ is the space $\Hom (V_1,V_2)$.  Of course it must be a $G$-equivariant section, and therefore the restriction to the basepoint $e\in K \backslash G$ determines the symbol.  This restriction is a smooth, $K$-equivariant function 
\[
a_0 \colon \partial (\mathfrak{g}/\mathfrak{k})^* \longrightarrow \Hom (V_1,V_2)
\]
(we are using the standard identification of $T_eM$ with $\mathfrak{g} / \mathfrak{k}$). 
Conversely, it is shown in  \cite{Stetkaer71} how to construct  an operator in $\Psi^0_{G,K}(V_1,V_2)$ with any given principal symbol function $a_0$, as above, by a simple averaging procedure (and we shall see the same averaging procedure below). 

\begin{definition}
We shall denote by  
\[
\Sigma^0_{G,K} (V_1,V_2)= C^\infty \bigl ( \partial (\mathfrak{g}/\mathfrak{k})^* , \Hom(V_1,V_2)\bigr )^K
\]
the space of principal symbols of operators in $\Psi^0_{G,K}(V_1,V_2)$.   
\end{definition} 

The discussion above provides a short exact sequence of vector spaces
\begin{equation}
\label{eq-pseudodifferential-extension-before-completion}
    0 \longrightarrow \Psi^{-1}_{G,K}(V_1,V_2) \longrightarrow \Psi^0_{G,K}(V_1,V_2) \longrightarrow \Sigma^0_{G,K} (V_1,V_2)  \longrightarrow 0
\end{equation} 
involving the space of properly supported, order $0$, classical  pseudodifferential operators acting between the spaces of smooth sections\footnote{The proper support condition implies that $A$ necessarily maps compactly supported sections to compactly supported sections. The proper support condition also allows one to extend  $A$ to an operator from all smooth sections to all smooth sections.} of $G$-equivariant complex vector bundles on the symmetric space $K \backslash G$.

\begin{proposition}
\label{prop-l-2-versus-s-star-boundedness}
  Every  properly supported, $G$-equivariant, order zero classical pseudodifferential operator 
\[
A \colon [C_c^\infty ( G)\otimes V_1]^K \longrightarrow  [C_c^\infty ( G)\otimes V_2]^K 
\]
extends to a bounded operator between $L^2$-spaces.
\end{proposition}

\begin{proof} 
Let $\varphi$ be a smooth and compactly supported function on $K \backslash G$ such that 
\[
\int_G  \varphi (xg) \, dg = 1 \quad \forall x\in K\backslash G,
\]
and let 
\[
A_e {=}  A\varphi  \colon [C_c^\infty ( G)\otimes V_1]^K \longrightarrow  [C_c^\infty ( G)\otimes V_2]^K .
\]
This is a compactly supported,  order zero classical pseudodifferential operator on $M{=}K\backslash G$, and therefore it is bounded in the $L^2$-operator norm. If we set $A_g = g(A_e)$, then 
\[
(A f)(x) = \int_G (A_g f)(x) \, dg \quad \forall x \in K \backslash G,\quad \forall f\in [C_c^\infty ( G)\otimes V_1]^K  .
\]
The operators $A_g$  are unitarily equivalent to one another, and therefore they have the same $L^2$-operator norm:
\[
\| A_g \| = \| A_e\| \quad  \forall g\in G.
\]
Let  $C\subseteq K \backslash G$ be a compact subset such that the  support of $A_e$ is included in 
$C{\times } C$, and let 
\[
S = \{ \, g\in G : Kg \in C\, \} .
\]
This is a compact subset of $G$, and if  $f\in[C_c^\infty ( G)\otimes V_1]^K $, $g\in G$ and $x\in K \backslash G$, then $(A_gf)(x) = 0$ unless $x\in Cg$. In other words, $(A_gf)(x) = 0$ unless $g\in x^{-1}S$. 
So by the Cauchy-Schwarz inequality,
\[
\begin{aligned}
\| A f\|_{L^2}^2 
    & = \int _M \Bigl \| \int_G (A_gf)(x)\, dg \Bigr\|_{V_2}^2 \, dx \\
    & \le \int _M \Bigl ( \int_{x^{-1} S} 1 \, dg\Bigr ) \Bigl (   \int_{x^{-1}S} \bigl \|(A_gf)(x)\bigr \|_{V_2} ^2 \, dg\Bigr )   \, dx \\
    & = \operatorname{vol}(S)\cdot   \int _M \Bigl (\int_{x^{-1}S} \bigl \|(A_gf)(x)\bigr \|_{V_2} ^2 \, dg \Bigr )   \, dx  .
\end{aligned}
\]
But now, using Fubini's theorem, we can write 
\[
\int _M \Bigl (\int_{x^{-1}S} \bigl \|(A_gf)(x)\bigr \|_{V_2} ^2 \, dg \Bigr )   \, dx \\
     = \int _G \Bigl (\int_{Cg}  \bigl \|(A_gf)(x)\bigr \|_{V_2} ^2 \, dx  \Bigr ) \, dg .
\]
The inner integral on the right may be estimated by
\[
\int_{Cg}  \bigl \|(A_gf)(x)\bigr \|_{V_2} ^2 \, dx \le \| A_e\|^2\cdot \int_{Cg}  \bigl \|f(x)\bigr \|_{V_1} ^2 \, dx ,
\]
and so,  by Fubini  once again,
\[
\begin{aligned}
\| A f\|_{L^2}^2  
    & \le  \operatorname{vol}(S)\cdot  \| A_e\|^2 \cdot 
    \int _G \Bigl (\int_{Cg}   \|f(x)  \|_{V_1} ^2 \, dx  \Bigr ) \, dg \\
    & = \operatorname{vol}(S)\cdot \| A_e\|^2  \cdot 
    \int _M \Bigl (\int_{x^{-1}S}   \|f(x)  \|_{V_1} ^2 \, dg \Bigr )   \, dx \\
    & = \operatorname{vol}(S)^2\cdot \| A_e\|^2  \cdot 
    \int _M    \|f(x)  \|_{V_1} ^2     \, dx
    \\
    & = \operatorname{vol}(S)^2 \cdot \| A_e\|^2 \cdot  \| f\|_{L^2}^2,
\end{aligned}
\]
which proves that $A$ is $L^2$-bounded, as required.
\end{proof} 

We noted in Theorem~\ref{thm-negative-order-operators-compact} that compactly supported,  negative order operators  are compact operators.  The analogue of this result in the current context is as follows:

\begin{proposition}
\label{prop-order-minus-one-is-compact}
    The $L^2$-operator norm-closure of the properly supported, equivariant smoothing operators 
    \[
    T \colon [C_c^\infty ( G)\otimes V_1]^K  \longrightarrow [C_c^\infty ( G)\otimes V_2]^K 
    \]
    includes all the properly supported, equivariant classical pseudodifferential  operators of order $-1$ or less.
\end{proposition}

\begin{proof} 
If $A$ is any pseudodifferential operator on $\R^n$ of order less than $-n$, then it follows from the integral formula \eqref{eq-fmla-for-psdo} that the Schwartz kernel of $A$ is a continuous function. It  follows that if $A$ is a properly supported and $G$-equivariant pseudodifferential operator, as in the statement of the proposition, and if the order of $A$ is less than $-\dim  (K\backslash G)$, then the Schwartz kernel of $A$ is an equivariant, continuous and properly supported  section of the bundle $\Hom( V_1,V_2)$ over $K\backslash G \times K\backslash G$.  So it corresponds to an element of 
\[
[  C_c  (G) \otimes \Hom (V_1,V_2)]^{K\times K} 
\]
under the isomorphism \eqref{eq-computational-form-of-equivariant-integral-kernels}, and $A$ is therefore in the operator-norm closure of the $G$-equivariant, properly supported smoothing operators.

Suppose now $A$ has order $-1$ or less. For $V= V_1\oplus V_2$ we can think of $A$ as an operator 
 \[
    A \colon [C_c^\infty ( G)\otimes V]^K  \longrightarrow [C_c^\infty ( G)\otimes V]^K 
    \]
that is zero on sections with values in $V_2$ and has range in the sections valued in $V_2$. The above shows that some power of $A^*A$ is in the operator-norm closure of the $G$-equivariant, properly supported smoothing operators.
But $A$ belongs to a $C^*$-algebra of operators on $[L^2( G)\otimes V]^K $ that includes the norm closure of the $G$-equivariant, properly supported smoothing operators as an ideal.  The above shows that the self-adjoint element $A^*A$ is nilpotent in the quotient $C^*$-algebra, and is therefore zero in the quotient $C^*$-algebra.
\end{proof}

\subsection{Pseudodifferential extension in the equivariant context}
\label{subsec-psdo-extension-in-equivariant-context}

\begin{definition} 
\label{def-three-c-star-categories}
Let $G$ be a  Lie  group, let $K$ be a  compact subgroup of $G$, and let $V_1$ and $V_2$ be finite-dimensional unitary representations of $K$.
\begin{enumerate}[\rm (i)]

\item Denote by  $\CC^*_{G,K}(V_1,V_2)$    the $L^2$-operator norm-closure of the vector space  $\Psi^{-1}_{G,K}(V_1,V_2)$, or equivalently the $L^2$-operator norm-closure of the space of properly supported, $G$-equivariant smoothing operators 
\[
[L^2 (G)\otimes V_1]^K\longrightarrow 
[L^2 (G)\otimes V_2]^K.
\]

\item Denote by   $\PSDO^*_{G,K}(V_1,V_2)$  the $L^2$-operator norm-closure of the space $\Psi^{0}_{G,K}(V_1,V_2)$.

\item Denote by 
\[
\Symb^*_{G,K} (V_1,V_2) = C\bigl (\partial (\mathfrak{g}/\mathfrak{k})^*, \Hom (V_1,V_2)\bigr )^K
\]
the closure in the supremum norm of the symbol space $\Sigma^0_{G,K} (V_1,V_2)$.  
\end{enumerate}
\end{definition}

\begin{lemma} 
Let $V_1$ and $V_2$ be finite-dimensional unitary representations of $K$ and let 
\[
A \colon [C_c^\infty ( G)\otimes V_1]^K  \longrightarrow 
[C_c^\infty ( G)\otimes V_2]^K 
\]
be a properly supported, $G$-equivariant, order zero, classical pseudodifferential operator. If $a_0$ is the principal symbol of $A$, then 
\[
\| A \|_{L^2\to L^2} \ge \sup_{\xi \ne 0} \bigl \| a_0 (e,\xi)\bigr \|.
\]
As a result, the principal symbol map extends to a bounded linear map 
\[
\PSDO^*_{G,K}( V_1,V_2) \longrightarrow \Symb^*_{G,K}( V_1,V_2).
\]
\end{lemma} 

\begin{proof} 
Let $\varphi$ be a smooth and compactly supported function on $K \backslash G$ such that $1 \ge \varphi\ge 0$ and $\varphi(e) =1$.  Then $\varphi A \varphi$ is a compactly supported classical pseudodifferential operator. Moreover 
$ \| A\|\ge \| \varphi A \varphi \| $ and the principal symbols of $A$ and $\varphi A \varphi$ agree at $e\in K\backslash G$. It follows from Theorem~\ref{thm-symbol-map-is-continuous} that 
\[
\| A\| \ge \| \varphi A \varphi \| \ge \sup_{\xi \ne 0} \bigl \| a_0 (e,\xi)\bigr \| ,
\]
as required.
\end{proof} 

\begin{proposition}
\label{prop-order-zero-psdo-extension}
Let $V_1$ and $V_2$ be finite-dimensional unitary representations of $K$. By taking $L^2$-operator norm closures, we obtain from the exact sequence 
 \[
    0 \longrightarrow \Psi^{-1}_{G,K}(V_1,V_2) \longrightarrow \Psi^0_{G,K}(V_1,V_2) \longrightarrow \Sigma^0_{G,K} (V_1,V_2)  \longrightarrow 0
    \]
an exact sequence    
    \[
    0 \longrightarrow \CC^*_{G,K}(V_1,V_2) \longrightarrow \PSDO^*_{G,K}(V_1,V_2) \longrightarrow \Symb^*_{G,K} (V_1,V_2)  \longrightarrow 0 .
    \]
\end{proposition}

To prove the proposition we shall use the following two lemmas.  The first may be proved using the method of the proof of Proposition~\ref{prop-l-2-versus-s-star-boundedness}; the second is straightforward.  Both proofs will be omitted.

\begin{lemma}
\label{lem-continuity-of-averaging}
    Let $\varphi$   be a continuous, compactly supported function  on $K \backslash G$.  If $V_1$ and $V_2$ are any finite-dimensional unitary representations of $K$, and if 
    \[
    T \colon [L^2 ( G)\otimes V_1]^K 
    \longrightarrow 
    [L^2 ( G)\otimes V_2]^K 
    \]
    is any bounded operator, then the averaging formula 
    \[
    \mathsf{Av}(\varphi T \varphi )s = \int _G g( \varphi T \varphi)s \, dg
    \]
    defines a bounded operator 
     \[
       \mathsf{Av}(\varphi T \varphi )\colon [L^2 ( G )\otimes V_1]^K
    \longrightarrow 
   [L^2 ( G )\otimes V_2]^K .
    \]
   Moreover the linear map $T \mapsto \mathsf{Av}(\varphi T \varphi )$ \textup{(}with $\varphi$  fixed\textup{)} is operator norm-continuous. \qed
\end{lemma}

\begin{lemma}
\label{lem-averaged-operator-has-same-principal-symbol}
  If $V_1$ and $V_2$ are any finite-dimensional unitary representations of $K$, and if 
    \[
    A \colon [C_c^\infty  (G)\otimes   V_1]^K 
    \longrightarrow 
  [C_c^\infty  (G)\otimes   V_2]^K  
    \]
    is any equivariant, properly supported, order zero classical pseudodifferential  operator, then  
    $ \mathsf{Av}(\varphi T \varphi )$ is also an equivariant, properly supported, order zero classical pseudodifferential  operator. Moreover if $\mathsf{Av}( \varphi^2)=\operatorname{id}$, then   $ \mathsf{Av}(\varphi T \varphi )$ has  the same principal symbol as $A$. \qed
\end{lemma}

\begin{proof}[Proof of Proposition~\ref{prop-order-zero-psdo-extension}]
If we set $V = V_1\oplus V_2$, then we can regard the space $\PSDO^*_{G,K}(V_1,V_2)$ as a subspace of $\PSDO^*(V,V)$, as in the proof of Proposition~\ref{prop-order-minus-one-is-compact}. 
We can do the same for the $\CC^*$- and $\Symb^*$-spaces, with the result that it suffices to prove exactness of the sequence 
\[
    0 \longrightarrow \CC^*_{G,K}(V,V) \longrightarrow \PSDO^*_{G,K}(V,V) \longrightarrow \Symb_{G,K}^* (V,V)  \longrightarrow 0 ,
\]
which is a sequence of $C^*$-algebras and $*$-homomorphisms.

The morphism from $\CC^*_{G,K}(V,V)$ to $\PSDO^*_{G,K}(V,V)$ is simply an inclusion, so it is certainly injective, and the morphism from $\PSDO^*_{G,K}(V,V)$  to $\Symb_{G,K}^*(V,V)$ is certainly surjective, since it has dense range, and the images of $C^*$-algebra morphisms are always closed. The continuity of the symbol mapping implies that the composition
\[
\CC^*_{G,K}(V,V) \longrightarrow \PSDO^*_{G,K}(V,V) \longrightarrow \Symb_{G,K}^* (V,V) 
\]
is zero.  So it only remains to show that this part of the sequence is exact. 

Suppose then that $T\in \PSDO^*_{G,K}(V,V)$ and that the principal symbol of $T$ is zero.  Write $T$ as an $L^2$-operator norm limit of pseudodifferential operators $A_n\in \Psi^0_{G,K}(V,V)$.  Then choose a smooth compactly supported function $\varphi$ on $K \backslash G$ such that $\mathsf{Av}(\varphi^2) = \operatorname{id}$,  and consider the operators 
\[
\mathsf{Av}(\varphi A_n \varphi) \colon [L^2  (G)\otimes   V]^K  \longrightarrow  [L^2  (G)\otimes   V]^K  .
\]
It follows from Lemma~\ref{lem-averaged-operator-has-same-principal-symbol} that the operators $\mathsf{Av}(\varphi A_n \varphi)$ are elements of $\Psi^0_{G,K}(V,V)$, and furthermore that 
\[
\mathsf{Av}(\varphi A_n\varphi) - A_n \in \Psi^{-1}_{G,K}(V,V),
\]
since the  principal symbol of the difference is zero.  It therefore follows from  Proposition~\ref{prop-order-minus-one-is-compact} that 
\begin{equation}
\label{eq-a-n-minus-average}
\mathsf{Av}(\varphi A_n\varphi) - A_n  \in \CC^*_{G,K}(V,V) 
\end{equation}
for all $n$.  We therefore find  that 
\begin{equation}
\label{eq-t-minus-average}
\mathsf{Av}(\varphi T\varphi) - T  \in \CC^*_{G,K}(V,V) ,
\end{equation}
since by Lemma~\ref{lem-continuity-of-averaging} the operator in  \eqref{eq-t-minus-average} is the norm limit of the  operators in \eqref{eq-a-n-minus-average}.  So to complete the proof if suffices to show that 
\begin{equation}
\label{eq-averaged-operator-in-c-star}
\mathsf{Av}(\varphi T \varphi)\in \CC^*_{G,K}(V,V).
\end{equation}
For this we may invoke  known facts about compactly supported pseudodifferential operators, as follows.

The compactly supported operator $\varphi T\varphi$ lies in the operator norm-closure  of the algebra of compactly supported, order zero classical pseudodifferential operators, and has vanishing principal symbol.  It is therefore a norm-limit of compactly supported smoothing operators $B_n$, by Theorem~\ref{thm-negative-order-operators-compact}.   Choosing a smooth, compactly supported function $\psi$ with $\psi \varphi = \varphi$, we now have 
\[
\lim_{n\to \infty}\bigl \| \mathsf{Av}( \varphi T \varphi) - \mathsf{Av}( \psi B_n \psi) \bigr \| 
= \lim_{n\to \infty} \bigl \| \mathsf{Av}( \psi \varphi T \varphi\psi ) - \mathsf{Av}( \psi B_n \psi) \bigr \| 
 =  0,
\]
where the second equality is a consequence of Lemma~\ref{lem-continuity-of-averaging}. This proves \eqref{eq-averaged-operator-in-c-star} and therefore the proposition.
\end{proof}

\section{C*-Categories}
\label{subsec-category-of-psdos}

Recall that a \emph{Banach category} is a category\footnote{In the context of Banach categories and $C^*$-categories, it is generally convenient \emph{not} to require that categories have identity morphisms, and we shall not do so here.}  for which each morphism space $\Hom(V,W)$ is equipped with the structure of a Banach space in such a way that the composition law is bilinear, with $\|S   T\|\le \| S\|\| T\|$, and recall that a \emph{$C^*$-category} is a Banach category that is equipped with an isometric and conjugate-linear involution operation on morphism spaces,
\[
\ast: \Hom(V,W) \longrightarrow \Hom(W,V),
\]
 such that
\[
(S   T)^{*} = T^{*}  S^{*}\quad \text{and} \quad  (T^{*})^{*} = T,
\]
for all composable morphisms $S$ and $T$, and in addition 
\[
T^*T\ge 0  \quad \text{and} \quad \| T^*T \| = \| T\|^2 ,
\]
for all $T$.  The meaning of the inequality above is that  the spectrum of the element $T^*T$  in the Banach algebra $\Hom (V,V)$ is a subset of $[0,\infty)$.  

The reader who is unfamiliar with $C^*$-categories may wish to consult \cite{MitchenerCategories02}, which offers an elementary introduction to the topic.  

The same reader may also wish to keep in mind the following construction. If $A$ is a $C^*$-algebra, and if $\{\, p_\sigma :  \sigma \in \Sigma\,\}$ is a family of pairwise orthogonal projections in the multiplier algebra of $A$ such that  $\sum _{\sigma \in \Sigma} p_\sigma = 1$, with convergence in the strict topology \cite[3.12.17]{PedersenCstarBook}, then there is an associated $C^*$-category $\mathsf{C}_A$ with
\begin{enumerate}[\rm (i)]

\item   objects  $\sigma \in \Sigma$, and 

\item  morphism spaces   $\mathsf{C}_A(\sigma_1,\sigma_2) =p_{\sigma_2}Ap_{\sigma_1}$.

\end{enumerate}
This process can be reversed, to assemble a $C^*$-algebra from a $C^*$-category, or at least from one that is  small enough;  see \cite[Sec.\,3]{Joachim03} for details.  Most of the $C^*$-categorical concepts to be discussed below are adapted from $C^*$-algebra concepts using this construction, and are best understood that way.  In particular, the $K$-theory groups of $A$ and $\mathsf{C}_A$, to be discussed in Section~\ref{subsec-k-theory-for-smoothing-category}, are canonically isomorphic \cite[Cor.\,3.5]{Joachim03}.

\subsection{C*-categories from pseudodifferential operators}
\label{sec-c-star-categories-from-psdo}
Here are the $C^*$-categories  that we wish to study.

\begin{definition} 
Let $G$ be a  Lie group, and let $K$ be a compact subgroup of $G$. We define $C^*$-categories $\PSDO^{*}_{G,K}$, $\CC^*_{G,K}$ and $\Symb^*_{G,K}$ as follows.
 \begin{enumerate}[\rm (i)]

\item For each of $\PSDO^{*}_{G,K}$, $\CC^*_{G,K}$ and $\Symb^*_{G,K}$, the objects are the finite-dimensional unitary representations of $K$.

 \item For  $\PSDO^{*}_{G,K}$, the space of morphisms from $V_1$ to $V_2$  is $\PSDO^*_{G,K}(V_1,V_2)$; for  $\CC^*_{G,K}$ it   is $\CC^*_{G,K}(V_1,V_2)$; and for  $\Symb^*_{G,K}$ it   is $\Symb^*_{G,K}(V_1,V_2)$.  
 
 \end{enumerate}
For each $C^*$-category, composition  and adjoint are the evident operations on operators or symbols. 
\end{definition}

The $C^*$-category $\CC^{*}_{G,K}$ is not only a subcategory of $\PSDO^*_{G,K}$ but also an \emph{ideal}, in the sense that the composition of any composable pair of morphisms in $\PSDO^*_{G,K}$ that includes at least one morphism from $\CC^{*}_{G,K}$ is a morphism in $\CC^{*}_{G,K}$; see \cite[Def.\,4.2]{MitchenerCategories02}.  In fact according to Proposition~\ref{prop-order-zero-psdo-extension}, the category $\CC^*_{G,K}$  is the \emph{kernel}  of the \emph{principal symbol functor} from $\PSDO^*_{G,K}$ to $\Symb^*_{G,K}$ (by kernel we mean the subcategory of $\PSDO^*_{G,K}$ comprised of all objects of $\PSDO^*_{G,K}$ and all morphisms that are mapped to $0$ by the symbol functor). Moreover $\Symb^*_{G,K}$ is the \emph{quotient $C^*$-category} $\PSDO^*_{G,K} / \CC^{*}_{G,K}$  in the sense of \cite[Def.\,4.4]{MitchenerCategories02}. We shall express this by saying that there is  an \emph{extension of $C^*$-categories}
\[
0 \longrightarrow \CC^{*}_{G,K} \longrightarrow \PSDO^{*}_{G,K} \longrightarrow \Symb^{*}_{G,K} \longrightarrow 0.
\]

\subsection{K-Theory for the   C*-category generated by smoothing operators} 
\label{subsec-k-theory-for-smoothing-category}
We turn now to $K$-theory. The $C^*$-categories that we defined in the previous section are all additive, and for an additive $C^*$-category $\mathsf A$ with unit morphisms, probably the simplest way to define the $K_0$-group is as the Grothendieck group of equivalence classes of idempotent morphisms in $\mathsf A$.\footnote{Equivalently,   $K_0(\mathsf{A})$ is the Grothendieck group of isomorphism classes of objects in the \emph{idempotent completion}, or \emph{Karoubi envelope}, of $\mathsf A$, which is constructed in \cite[Thm.\,I.6.10]{KaroubiKTheory}.}  If $\mathsf A$ does not have unit morphisms, one may embed $\mathsf A$ as an ideal in a larger $C^*$-category $\mathsf B$ with unit morphisms, and then define $K_0(\mathsf {A})$ as the relative $K_0$ group for the function $\mathsf{B}\to \mathsf{B}/\mathsf{A}$; compare \cite[Sec.\,II.2]{KaroubiKTheory}. Then one may define $K_1$ using a suspension construction, in which a morphism in the suspended category is a path of morphisms in the original category that begins and ends at the zero morphism. See \cite[Sec.\,II.3]{KaroubiKTheory}.

All this is exactly analogous to how one may define $K$-theory for $C^*$-algebras.  For other approaches, which are analogous to various other approaches to $C^*$-algebra $K$-theory, see \cite{MitchenerKTheory01,Mitchener02,Joachim03}.

In this section, we shall prove that the $K$-theory groups of the $C^*$-category $\CC^*_{G,K}$ coincide with those  of $C^{*}_r(G)$.

\begin{definition}
\label{def-p-s}
    If $S \subseteq \widehat{K}$ is a finite subset, and if $H$ is any Hilbert space that is equipped with a unitary representation of $K$, then we shall denote by $P_S\colon H\to H$ the 
 operator of  convolution over $K$ with the function
\[
k\longmapsto  \sum_{\pi \in S} \frac{\dim(\pi)}{\mu(K)}  \chi_{\pi}(k^{-1}), 
\]
where $\chi_\pi$ is the character of $\pi$, so that $P_S$ is the orthogonal projection onto the direct sum of the $\pi$-isotypical components of $H$, for $\pi \in S$.  In addition, we shall  denote by  
\[
C^*_r(G;S)\subseteq \mathfrak{B}(L^2 (G))
\]
the compression $P_SC^*_r(G)P_S$, where we let $K$ act on $L^2(G)$ by restriction of the left-regular representation of $G$. This is a $C^*$-subalgebra of $C^*_r(G)$. 
\end{definition}

\begin{lemma}
\label{lem-c-star-algebra-k-theory-as-direct-limit}
    The inclusions of all $C^*_r (G;S)$ into $C^*_r (G)$ induce an isomorphism 
\[
\varinjlim_{S} K_{*}(C^*_r(G;S)) \stackrel \cong \longrightarrow K_{*}(C^*_r(G))
\]
in $K$-theory.
\end{lemma}

\begin{proof}
The union of $C^*$-subalgebras $\bigcup_{S} C^*_{r}(G;S)$ is a dense subalgebra of $C^*_r(G)$, so the conclusion follows from continuity of $K$-theory,   as in for example 
\cite[Thm.\ 6.3.2]{RordamEtAl} or \cite[Ex.\ II.6.15]{KaroubiKTheory}.
\end{proof}

A similar result holds for the $C^*$-category $\CC^*_{G,K}$. 

\begin{definition}
   If $S \subseteq \widehat{K}$ is any finite subset, then denote by $\CC^*_{G,K}(S)$ the full additive $C^*$-subcategory of $\CC^*_{G,K}$ on those objects $V$ whose $K$-isotypical decompositions only include representations from $S$.  
\end{definition}

\begin{lemma}
\label{lem-c-star-category-k-theory-as-direct-limit}
    The inclusions of the subcategories $\CC^*_{G,K}(S)$ into $\CC^*_{G,K}$ induce an isomorphism
    \[
    \varinjlim K_{*}(\CC^*_{G,K}(S)) \stackrel \cong \longrightarrow K_{*}(\CC^*_{G,K})
    \]
    in $K$-theory.
\end{lemma}

\begin{proof}
 Clear because $\CC^*_{G,K}$ is the union of all the subcategories $\CC^*_{G,K}(S)$.
\end{proof}

Now, the projection $P_S$ in Definition~\ref{def-p-s} is (the operator of convolution with) a smooth function on $K$, and a central projection in $C^*_r(K)$, and the inclusion morphisms induce vector space isomorphisms
\begin{equation}
    \label{eq-c-star-equals-smooth-equals-l-2}
P_S C^*_r (K) \stackrel\cong\longleftarrow P_S C^\infty (K) \stackrel \cong \longrightarrow P_SL^2(K).
\end{equation}
Let us write 
\[
V_S = P_S L^2(K);
\]
this is a finite-dimensional unitary representation of $K$. The inner product on $V_S$ is transported, via the isomorphisms in \eqref{eq-c-star-equals-smooth-equals-l-2}, to the inner product
\[
\langle f,h\rangle = (f^*h)(e)
\]
on $P_S C^*_r (K)$, and the correspondence  
\[
C^*_r(K)P_S \ni f \longmapsto \bigl [  h\mapsto \langle f^*, h\rangle \bigr] \in V_S^*
\]
identifies $C^*_r(K)P_S$ with the  dual of $V_S$, as unitary representations of $K$. From this we obtain  a $K$-bi-equivariant isomorphism
\[
P_SC^*_r (K)\otimes C^*_r(K)P_S \stackrel \cong \longrightarrow 
 V_S \otimes V_S^* \stackrel \cong \longrightarrow \End (V_S).
\]
From this we obtain, by rearranging the tensor factors, a linear isomorphism
\begin{equation}
    \label{eq-kgk-to-gkk-isomorphism}
\left[P_S C^*_r(K) \otimes C^*_r(G) \otimes C^*_r(K)P_S\right]^{K \times K} \stackrel \cong\longrightarrow  \left[ C^*_r(G) \otimes \mathrm{End}(V_S) \right]^{K \times K},
\end{equation}
where, on the left-hand side, the first factor of $K$ acts on the left of $P_S C^*_r (K)$ and on the left of $C^*_r(G)$, while the second factor of $K$ acts on the right  of $C^*_r(G)$ and on the right of $C^*_r (K)P_S$. 

The left-hand side in \eqref{eq-kgk-to-gkk-isomorphism} has a natural $*$-algebra structure, for which the  product is
\[
(f_1 \otimes f_2 \otimes f_3) \cdot (h_1 \otimes h_2 \otimes h_3) = f_1 \otimes (f_2  f_3  h_1  h_2) \otimes h_3
\]
and   the  $*$-operation  is
\[
\pushQED{\qed}
(f_1 \otimes f_2 \otimes f_3)^* =  (f_3^* \otimes f_2^* \otimes f_1^*)  .
\]

\begin{lemma}
\label{lem-c-star-algebra-as-endomorphism-algebra}
The formula 
\[
(f_1 \otimes f_2 \otimes f_3) \mapsto f_1  f_2   f_3
\]
defines an isomorphism of $*$-algebras
\[
\pushQED{\qed}
\left[P_S C^*_r(K) \otimes C^*_r(G) \otimes C^*_r(K)P_S\right]^{K \times K} \stackrel \cong   \longrightarrow  P_SC^*_r(G)P_S.
\qedhere
\popQED
\]
\end{lemma}

We  obtain from \eqref{eq-kgk-to-gkk-isomorphism} and Lemma~\ref{lem-c-star-algebra-as-endomorphism-algebra} an isomorphism of $*$-algebras
\begin{equation}
    \label{eq-star-algebra-isomorphism}
C^*_r(G;S) \stackrel \cong \longrightarrow  \left[ C^*_r(G) \otimes \mathrm{End}(V_S) \right]^{K \times K} ,
\end{equation} 
which is necessarily also an isomorphism of $C^*$-algebras.
Now, we noted in \eqref{eq-norm-completion-f-smoothing-ops} that the right-hand side of \eqref{eq-star-algebra-isomorphism} is the endomorphism $C^*$-algebra $\CC^*_{G,K}(V_S,V_S)$ in the $C^*$-category $\CC^*_{G,K}$.  So since each $C^*$-algebra is a $C^*$-category on one object, the isomorphism \eqref{eq-star-algebra-isomorphism} can be viewed as a functor between $C^*$-categories, and it therefore induces  a morphism of $K$-theory groups
\begin{equation}
    \label{eq-k-theory-morphism-for-each-s}
K_*(C^*_r(G;S)) \longrightarrow K_*(\CC^*_{G,K}(S)).
\end{equation}
This is indeed an isomorphism, because   $V_S$ is a finite projective generator for the category $\CC^*_{G,K}(S)$.

\begin{theorem}
\label{thm-k-theory-of-smoothing-op-category}
Let $G$ be a Lie group and let $K$ be a compact subgroup of $G$. There is a unique isomorphism of $K$-theory groups 
\[
  K_*(C_r^*(G)) \stackrel \cong \longrightarrow K_*(\CC^*_{G,K})
\]
such that, for every $S$, the  diagram 
\[
\xymatrix@C=40pt{
K_*(C^*_r(G;S)) \ar[r]^{\eqref{eq-k-theory-morphism-for-each-s}} \ar[d] & K_*(\CC^*_{G,K}(S))\ar[d]
\\
K_*(C^*_r(G))\ar[r]  & K_*(\CC^*_{G,K})
}
\]
is commutative  \textup{(}the vertical arrows are induced from inclusions\textup{)}.
\end{theorem}

\begin{proof}
The morphisms \eqref{eq-k-theory-morphism-for-each-s} are compatible with inclusions $S_1\subseteq S_2$ of finite subsets of $\widehat K$.   So  the existence and uniqueness parts of the theorem both follow from Lemmas \ref{lem-c-star-algebra-k-theory-as-direct-limit} and \ref{lem-c-star-category-k-theory-as-direct-limit}.
\end{proof}

\begin{remark}
The arguments above also prove that  the $C^*$-category $\CC^*_{G,K}$ is equivalent to the additive completion (\cite{MitchenerKTheory01}, Def.2.12)  of $\mathsf{C}_A$ described at the beginning of Section~\ref{subsec-category-of-psdos} that is constructed from the $C^*$-algebra $A{=}C^*_r(G)$, and from the family of projection operators $P_\pi$ associated to the irreducible representations of $K$ in Definition~\ref{def-p-s}.
\end{remark}

\section{The Connes-Kasparov isomorphism}

We shall now specialize further,  and work with an almost-connected Lie group $G$ and maximal compact subgroup $K$.\footnote{For the existence and uniqueness, up to conjugacy, of maximal compact subgroups of almost-connected Lie groups, see \cite[Thm.\,XV.3.1]{HochschildStructureOfLieGroups}.}

\subsection{The deformation to the normal cone and its C*-algebra} 
The \emph{motion group} associated to $G$ and $K$ is the semidirect product group 
\[
G_0 = K \ltimes ( \mathfrak{g}/\mathfrak{k} )
\]
constructed from the adjoint action of $K$ on the quotient vector space $\mathfrak{g}/\mathfrak{k}$. The  motion group fits into a smooth, one-parameter family of groups $\{ G_t \} _{t\in \R}$ with 
\[
G_t = 
\begin{cases} 
G & t\ne 0 \\
K \ltimes ( \mathfrak{g}/\mathfrak{k} ) & t=0.
\end{cases}
\]
That is, there is a smooth manifold $\GG $ and a submersion  $\GG\to \R$ whose fibers are these groups, for which the fiberwise-defined group operations of multiplication, inverse, and inclusion of the group identity element,
\[
\GG  \times _{\R} \GG  
\longrightarrow \GG , 
\quad 
\GG  
\longrightarrow \GG 
\quad \text{and}\quad 
\mathbb{\R} 
\longrightarrow \GG 
\]
are smooth maps. This is an instance of the deformation to the normal cone construction; see \cite[Sec.\,2.4]{DebordSkandalis19} for a concise introduction to the topic.  

The smooth manifold structure on $\GG$ may be characterized as follows. First, the complement of $G_0$ is an open subset, and carries the standard smooth structure of $G\times \R^{\times}$.  Second, if $V$ is any finite-dimensional real vector space, and if $E\colon V \to G$ is any smooth map  such that 
\begin{enumerate}[\rm (i)]

\item $E(0) = e$, and 

\item the map 
\[
\begin{gathered}
K\times V \longrightarrow G
\\
(k,X) \longmapsto k \cdot E(X)
\end{gathered}
\]
is a diffeomorphism onto an open subset of $G$, 
\end{enumerate}
then the map  
\[
K \times V  \times \R \longrightarrow \pmb{G}
\]
defined by 
\[
(k,v,t) \longmapsto \begin{cases} k \cdot E(tv) \in G_t & t \ne 0 \\
(k,dE(v))\in G_0 & t=0
\end{cases}
\]
is a diffeomorphism onto an open subset of $\pmb{G}$.  Here $dE \colon V \to \mathfrak{g} / \mathfrak{k}$ is the composition of the derivative of the smooth map $E$ at $0\in V$ with the projection map from $\mathfrak{g} $ to the quotient vector space $\mathfrak{g}/\mathfrak{k}$.

Associated to $\GG $ is a continuous field of $C^*$-algebras $\{ C^*_r (G_t)\} _{t\in \R}$ (see \cite[Ch.\,10]{Dixmier1977} for background information on continuous fields).  It is constructed by choosing a smoothly varying family of Haar measures on the fiber groups $G_t$, and then decreeing that the continuous sections of the continuous field  be generated by the smooth, compactly supported functions on $\GG $: from each such function $f$ one obtains by restriction a family of functions $f_t$ on the  groups $G_t$, and it is proved in \cite[Lemma\,6.13]{Higson08} that the norms $\| f_t\|_{C^*_r (G_t)}$ vary continuously with $t$.

 \begin{definition} 
 We shall denote by $C^*_r (\GG )$ the $C^*$-algebra of continuous sections of $\{ C^*_r (G_t)\} $ over the closed interval $[0,1] \subseteq \R$.
 \end{definition}

The following difficult and important result will be essential to everything that follows: 

\begin{theorem}[Connes-Kasparov isomorphism]
\label{thm-dnc-formulation-of-connes-kasparov}
    Let $G$ be an almost-con\-nected Lie group with maximal compact subgroup $K$. The $C^*$-algebra morphisms $C^*_r(\GG ) \to C^*_r (G_t)$  given by evaluation of continuous sections at any $t\in [0,1]$ are isomorphisms in $K$-theory.
\end{theorem}

The case $t{=}0$ is elementary, but for all other $t{\in} [0,1]$ the theorem  is by no means trivial.  In fact Connes has pointed out in  \cite[Prop.~9,~p.141]{ConnesNCG} that assertion in the theorem is equivalent to the Connes-Kasparov isomorphism.

\subsection{K-theory for the  C*-category of pseudodifferential operators} 
\label{sec-k-theory-for-psdo-category}
In this section we shall reformulate Theorem~\ref{thm-dnc-formulation-of-connes-kasparov}, or equivalently the Connes-Kasparov isomorphism, as an assertion about the $K$-theory of the $C^*$-cat\-egory of pseudodifferential operators that we have constructed in this paper.

\begin{definition}
Let $K$ be a compact Lie group. We shall denote by $\Rep^*_K$ the $C^*$-category whose objects are the finite-dimen\-sional unitary representations of   $K$ and whose morphisms are the $K$-equivariant linear maps between representations.
\end{definition} 

\begin{definition} 
\label{def-connes-kasparov-functor}
Let $G$ be an almost-connected Lie group and let $K$ be a maximal compact subgroup of $G$. We shall denote by  
\[
\CoKa \colon \Rep^*_K \longrightarrow \PSDO^*_{G,K}
\]
the functor of $C^*$-categories that is the identity on objects, and   maps a morphism of representations 
$T\colon V_1\to V_2$ to the induced morphism 
\[
T \colon [L^2( G)\otimes V_1]^K \longrightarrow  [L^2( G)\otimes V_2]^K
\]
(which is a particularly simple example of an equivariant, properly supported, classical order $0$ pseudodifferential operator). 
\end{definition}

The name $\CoKa$ is a reference to Connes and Kasparov, for we shall prove that Theorem~\ref{thm-dnc-formulation-of-connes-kasparov} implies, and indeed is equivalent to, the following assertion: 

\begin{theorem}
\label{thm-k-theory-of-the-c-star-category}
Let  $G$ be an almost-connected group with maximal compact subgroup $K$. The functor $\CoKa\colon \Rep^*_K \to \PSDO^*_{G,K}$ induces an isomorphism in $K$-theory.
\end{theorem}

The main idea of the proof is to reduce to the case in which  $G$   is replaced by its   motion group $G_0$. To make this reduction  we shall avail ourselves of the smooth family of groups $\GG=\{ G_t\}$ and the associated continuous field of $C^*$-algebras $\{ C^*_r (G_t)\}$.  We shall begin preparations for the proof  by considering not order zero pseudodifferential operators, but smoothing operators.  

\begin{definition}
Let $G$ be an almost-connected Lie group with maximal compact subgroup $K$.
We shall denote by $\CC^*_{\GG,\KK}$ the $C^*$-category  whose objects are the finite-dimensional unitary representations of $K$, and whose morphism spaces  $\CC^*_{\GG,\KK}(V_1,V_2)$  are norm closures of the spaces  of properly supported smooth families of equivariant smoothing operators 
\[
T_t \colon L^2 (K \backslash G_t, V_1)  \longrightarrow  L^2(K \backslash G_t, V_1)\qquad (t\in [0,1])
\]
(see below) in the norm $\| T\| = \sup\{\, \|T_t\|: t\in [0,1]\,\}$.
\end{definition}

In order to be more specific about the term \emph{properly supported smooth family}, we use the description of individual properly supported and equivariant smoothing operators as elements of the spaces 
\begin{equation}
    \label{eq-where-smoothing-operators-live}
\bigl [C_c^\infty (G_t)\otimes \Hom (V_1,V_2) \bigr ]^{K\times K}
\end{equation}
in Section~\ref{subsec-smoothing-operators}.  Our requirement on  properly supported, smooth families $\{ T_t\}$ of equivariant smoothing operators   is that there should be a single element in the space 
\[
\bigl [  C_c^\infty (\GG) \otimes \Hom (V_1,V_2)\bigr ]^{K\times K}
\]
that restricts at each $t\in [0,1]$ to the element in \eqref{eq-where-smoothing-operators-live} corresponding to $T_t$. 

With this definition, it is evident that there is an isomorphism of $C^*$-categories that is the identity on objects, and on morphism spaces takes the form of isomorphisms 
\[
\CC^*_{\GG,\KK}(V_1,V_2) \stackrel \cong \longrightarrow \bigl [   C^*_r (\GG) \otimes \Hom (V_1,V_2)\bigr ]^{K\times K}.
\]

\begin{theorem}
\label{thm-c-k-isomorphism-stated-using-c-star-categories}
The Connes-Kasparov isomorphism in Theorem~\textup{\ref{thm-dnc-formulation-of-connes-kasparov}} is equivalent to the assertion that for all $t\in [0,1]$ the functor of evaluation at $t$, 
\[
\CC^*_{\GG,\KK} \longrightarrow \CC^*_{G_t, K},
\]
induces an isomorphism in $K$-theory.  
\end{theorem}

\begin{proof}
    We  may repeat the arguments of Section~\ref{subsec-k-theory-for-smoothing-category} using the smooth family $\pmb{G}$ over $[0,1]$ in place of the single group $G$.  So for a finite subset $S\subseteq \widehat K$ we may define the $C^*$-algebra $C^*_r(\pmb{G};S)$, the $C^*$-category $\CC^*_{\pmb{G},\pmb{K}}(S)$, and the $K$-theory isomorphism
    \begin{equation*}
K_*(C^*_r(\pmb{G};S)) \stackrel\cong \longrightarrow K_*(\CC^*_{\pmb{G},\pmb{K}}(S)).
\end{equation*}
analogous to \eqref{eq-k-theory-morphism-for-each-s}.   These determine an isomorphism
\begin{equation}
    \label{eq-k-theory-for-smoothing-families}
K_*(C^*_r(\pmb{G})) \stackrel\cong \longrightarrow K_*(\CC^*_{\pmb{G},\pmb{K}})
\end{equation}
analogous to the isomorphism in Theorem~\ref{thm-k-theory-of-smoothing-op-category}.  Now, \eqref{eq-k-theory-for-smoothing-families} is compatible with evaluation at $t\in [0,1]$, in the sense that the diagram 
\[
\xymatrix@C=40pt{
K_*(C^*_r(\pmb{G})) \ar[r]^{\eqref{eq-k-theory-for-smoothing-families}}_{\cong}\ar[d]_{\text{eval.~at $t$}}&  K_*(\CC^*_{\pmb{G},\pmb{K}})\ar[d]^{\text{eval.~at $t$}}
\\
K_*(C^*_r(G_t)) \ar[r]_{\text{Thm\,\ref{thm-k-theory-of-smoothing-op-category}}}^{\cong}  & K_*(\CC^*_{G_t,K})
}
\]
commutes.  The theorem follows from this.
\end{proof}

The proof of Theorem~\ref{thm-k-theory-of-the-c-star-category} will  require a similar treatment of the order zero pseudodifferential operators: we shall need to organize families of order zero operators on the family of homogeneous spaces $K \backslash G_t$ into a $C^*$-category, $\PSDO^*_{\GG,\KK}$.  This is a bit more complicated, and to begin we shall simply assert the existence of a suitable category, and  explain how it may be used to prove Theorem~\ref{thm-k-theory-of-the-c-star-category}.  We shall actually construct $\PSDO^*_{\GG,\KK}$ in Section~\ref{subsec-construction-of-technical-c-star-category}.

 As usual, the objects of $\PSDO^*_{\GG,\KK}$ will be the finite-dimensional unitary representations of $K$.  The morphisms will be suitable continuous families of equivariant, properly supported, order zero classical pseudodifferential operators 
\begin{equation}
    \label{eq-family-of-operators-on-fibers-of-dnc}
A_t \colon  [ L^2 (G_t)\otimes  V_1]^K \longrightarrow  [ L^2 (G_t)\otimes  V_2]^K\qquad (t\in [0,1]),
\end{equation}
or norm-completions of such families.  Here are the properties of the category $\PSDO^*_{\GG,\KK}$ that we shall require for the argument: 

\begin{enumerate}[\rm (i)]

\item The category $\PSDO^*_{\GG,\KK}$ includes the category $\CC^*_{\GG,\KK}$  as an ideal.

\item The individual  operators in a morphism $\{ A_t : t\in [0,1]\}$  in each family all have the same principal symbol (observe that for all $t$, the tangent space of $K \backslash G_t$ identifies naturally with $\mathfrak{g} / \mathfrak{k}$, so that the principal symbol is, for all $t$, a continuous function on $\partial( \mathfrak{g} / \mathfrak{k})^*$).

\item The principal symbol functor (that associates to a family of operators the common principal symbol) fits into an extension of $C^*$-categories
\[
0 \longrightarrow \CC^*_{\GG,\KK} \longrightarrow \PSDO^*_{\GG,\KK}\longrightarrow  \Symb^*_{G,K} \longrightarrow 0
\]

\item If $T\colon V_1 \to V_2$ is a morphism of finite-dimensional unitary representations of $K$, then the family of induced morphisms 
\[
T_t \colon [ L^2 (G_t)\otimes  V_1]^K  \longrightarrow [ L^2 (G_t)\otimes  V_2]^K  \qquad (t \in [0,1])
\]
is a morphism in $\PSDO^*_{\GG,\KK}$. 

\end{enumerate}

\begin{theorem}
\label{thm-existence-of-the-technical-c-star-category} 
There exists a category $\PSDO^*_{\GG,\KK}$ with the properties \textup{(i)-(iv)} above.    
\end{theorem}

Taking this for granted, for now, we shall proceed to study the $K$-theory of the category $\PSDO^*_{\GG,\KK}$. 

\begin{theorem} 
\label{thm-c-k-isomorphism-in-psdo-form}
The Connes-Kasparov isomorphism in Theorem~\textup{\ref{thm-dnc-formulation-of-connes-kasparov}} is equivalent to the assertion that for all $t\in [0,1]$ the functor of evaluation at $t$, 
\[
\PSDO^*_{\GG,\KK} \longrightarrow \PSDO^*_{G_t, K},
\]
induces an isomorphism in $K$-theory.  
\end{theorem} 

\begin{proof}
This follows from Theorem~\ref{thm-c-k-isomorphism-stated-using-c-star-categories} and the five lemma, applied to the diagram of $K$-theory six-term exact sequences associated to the diagram 
\[
\xymatrix{
0 \ar[r]&  \CC^*_{\GG,\KK} \ar[r] \ar[d] & \PSDO^*_{\GG,\KK}\ar[r]\ar[d] &  \Symb^*_{G,K} \ar[r]\ar@{=}[d] &  0
\\
0 \ar[r]& \CC^*_{G_t,K} \ar[r]& \PSDO^*_{G_t,K} \ar[r]&  \Symb^*_{G,K} \ar[r]&  0
}
\]
where the vertical maps are given by evaluation at $t\in [0,1]$.
\end{proof}

Let us  now  examine the category $\PSDO^*_{G_0,K}$. 

\begin{theorem}
\label{thm-k-theory-of-the-motion-group-c-star-category}
Let $G_0$ be the  motion group associated to an almost-connected Lie group $G$ with maximal compact subgroup $K$.  The functor $\CoKa\colon \Rep^*_K \to \PSDO^*_{G_0,K}$ induces an isomorphism in $K$-theory.
\end{theorem}

\begin{proof} 
In the case of the motion group, the homogeneous space $K\backslash G_0$ identifies with the vector space $\mathfrak{g} / \mathfrak{k}$, and the $G_0$-equivariant operators 
  \[
    T \colon [C_c^\infty (  G_0 )\otimes  V_1]^K \longrightarrow [C_c^\infty (  G_0 )\otimes  V_2]^K
 \]
are precisely the operators 
\[
 T \colon C_c^\infty (\mathfrak{g}/\mathfrak{k})\otimes  V_1 \longrightarrow C_c^\infty (\mathfrak{g}/\mathfrak{k}) \otimes  V_2
\]
that are both $K$-equivariant and translation-invariant. 

Being translation-invariant, the operators $T$ above may be analyzed using the Fourier transform. We find that the properly supported, $G_0$-equi\-variant, order $0$ classical pseudodifferential operators $T$ as above are precisely the operators of convolution by $\Hom (V_1,V_2)$-valued, compactly supported distributions on $\mathfrak{g}/\mathfrak{k}$ whose Fourier transforms are $K$-equivariant, order zero classical symbol functions on $(\mathfrak{g}/\mathfrak{k})^*$ with values in $\Hom (V_1,V_2)$. 

Now, every order zero, classical symbol function extends by continuity to the compactification  of $(\mathfrak{g}/\mathfrak{k})^*$ in Definition~\ref{def-spherical-compactification}, and these extended functions separate the points of the compactification.
This yields an isomorphism of $C^*$-categories from  $\PSDO^*_{G_0,K}$ to the $C^*$-category whose objects are the finite-dimen\-sional unitary representations of $K$, and whose morphisms are elements of the spaces 
\[
C \bigl (\overline {(\mathfrak{g} / \mathfrak{k})^*} , \Hom (V_1,V_2)\bigr )^K .
\]
  The functor $\CoKa$ corresponds to the inclusion 
\[
\Hom (V_1,V_2)^K \longrightarrow C \bigl (\overline {(\mathfrak{g} / \mathfrak{k})^*} , \Hom (V_1,V_2)\bigr )^K 
\]
as constant functions.  It has a left-inverse 
\[
C \bigl (\overline {(\mathfrak{g} / \mathfrak{k})^*} , \Hom (V_1,V_2)\bigr )^K\longrightarrow  \Hom (V_1,V_2)^K 
\]
given by evaluation of functions at $0$.  The composition of the two functors in the opposite order is the  functor 
\[
F\colon C \bigl (\overline {(\mathfrak{g} / \mathfrak{k})^*} , \Hom (V_1,V_2)\bigr )^K  \longrightarrow C \bigl (\overline {(\mathfrak{g} / \mathfrak{k})^*} , \Hom (V_1,V_2)\bigr )^K 
\]
that replaces each function by the constant function with the same value at $0$.  Because the compactification $\overline {(\mathfrak{g} / \mathfrak{k})^*}$ is $K$-equivariantly contractible, the composition is homotopic through a family of functors $F_t$ to the identity functor (the continuity condition on the family is that for every morphism $a$, $F_t(a)$ is a norm-continuous function of $t$).  It therefore follows from the homotopy-invariance of  $K$-theory that $F$ induces the identity map on $K$-theory.  Thus $\CoKa$ induces an isomorphism on $K$-theory, as required.
\end{proof}

With this, we can complete the proof of Theorem~\ref{thm-k-theory-of-the-c-star-category}.  By the  assumptions (i) and (iii) above, the    functor 
\begin{equation}
    \label{eq-eval-at-zero-functor}
  \PSDO^*_{\GG,\KK} \longrightarrow 
 \PSDO^*_{G_0,K}  
\end{equation}
of evaluation at $t{=}0$ is surjective on Hom-spaces. Its kernel---the ideal in $\PSDO^*_{\GG,\KK}$ of comprised of all morphisms that map to a zero morphism---consists of those  operator families that define morphisms in $\PSDO^*_{\GG,\KK}$ and that vanish at $t{=}0$.  In particular, the principal symbols of such families vanish, and so by (iii) the families are morphisms in $\CC^*_{\GG,\KK}$.  Therefore the kernel is precisely the ideal in  $\CC^*_{\GG,\KK}$ whose morphisms are the  families that vanish at $t{=}0$.  This ideal is contractible in the sense of  \cite[Sec.\,3]{MitchenerKTheory01} and so has vanishing $K$-theory \cite[Prop.\,3.23]{MitchenerKTheory01}. Therefore, by the $K$-theory long-exact sequence associated to an extension of $C^*$-categories \cite[Cor.\,3.20]{MitchenerKTheory01}, the functor \eqref{eq-eval-at-zero-functor} induces an isomorphism in $K$-theory.

Next, we may define a  functor 
\begin{equation}
    \label{eq-families-ck-functor}
\pmb{\CoKa}\colon \Rep^*_K \longrightarrow \PSDO^*_{\GG,\KK} 
\end{equation}
exactly as we defined the original version in Definition~\ref{def-connes-kasparov-functor}, but using families, and using property (iv) of the category $\PSDO^* _{\GG,\KK}$ listed above. The composition 
\begin{equation}
    \label{eq-composition-with-eval-at-0}
\xymatrix@C=40pt{
 \Rep^*_K   \ar[r]^{\pmb{\CoKa}} &    \PSDO^*_{\GG,\KK} 
\ar[r]^-{\eqref{eq-eval-at-zero-functor}} &  K_* (\PSDO^*_{G_0,K} ) 
}
\end{equation}
is precisely the morphism that was proved in Theorem~\ref{thm-k-theory-of-the-motion-group-c-star-category} to be an isomorphism at the $K$-theory level.  So  the functor \eqref{eq-families-ck-functor} induces an isomorphism on $K$-theory, too.

So far we have not used the Connes-Kasparov isomorphism in Theorem~\ref{thm-dnc-formulation-of-connes-kasparov}.  But if we now use it in the equivalent form presented in Theorem~\ref{thm-c-k-isomorphism-in-psdo-form}, then we may conclude that both morphisms in the composition
\begin{equation}
    \label{eq-composition-with-eval-at-t}
\xymatrix@C=40pt{
K_*(\Rep^*_K  ) \ar[r]^{\pmb{\CoKa}} &   K_*(\PSDO^*_{\GG,\KK} ) 
\ar[r]^{\text{eval.\ at $t$}} &  K_* (\PSDO^*_{G_t,K} ) 
}
\end{equation}
are isomorphisms for all $t \in [0,1]$.  But when $t \neq 0$, the composition \eqref{eq-composition-with-eval-at-t} is precisely the morphism in the statement of Theorem~\ref{thm-k-theory-of-the-c-star-category}, and so the proof of Theorem~\ref{thm-k-theory-of-the-c-star-category} is complete. 

In the reverse direction, if we assume Theorem~\ref{thm-k-theory-of-the-c-star-category}, which is to say if we assume that the composition in \eqref{eq-composition-with-eval-at-t} is an isomorphism when $t \neq 0$, then since we already proved, without recourse to the Connes-Kasparov isomorphism, that \eqref{eq-families-ck-functor} induces an isomorphism on $K$-theory, it follows that evaluation at $t$ induces an isomorphism
\[
 K_*(\PSDO^*_{\GG,\KK} ) \stackrel \cong \longrightarrow  K_* (\PSDO^*_{G_t,K} ) 
\]
for all $t\in [0,1]$ (the case $t = 0$ is covered by Theorem~\ref{thm-k-theory-of-the-motion-group-c-star-category}), and according to Theorem~\ref{thm-c-k-isomorphism-in-psdo-form} this implies the Connes-Kasparov isomorphism.

\subsection{Construction of a C*-category from families of pseudodifferential operators}
\label{subsec-construction-of-technical-c-star-category}

In this section we shall construct the category $\PSDO_{\GG,\KK}$ that is used in the proof of Theorem~\ref{thm-k-theory-of-the-c-star-category}. 

The main issue is to decide when a family of operators \eqref{eq-family-of-operators-on-fibers-of-dnc} acting on the fibers of the smooth family of homogeneous spaces  $\pmb K \backslash \pmb G$ should be deemed to be continuous.  There are a number of equivalent ways of expressing the continuity condition, and  we shall choose the one that seems to us to be the quickest (what seems to us to be the most conceptual approach will be outlined in the next section).  The approach that we shall follow is based on the fact, first documented in \cite{AtiyahSingerFour}, that if one conjugates a compactly supported, order zero pseudodifferential operator  by a smooth family of diffeomorphisms, then one obtains an $L^2$-operator norm-continuous family of bounded Hilbert space operators: 

\begin{theorem}[See for instance {\cite[Prop.~1.3]{AtiyahSingerFour}}]
\label{thm-norm-continuity-of-psdos-under-conjgation-by-families-of-diffeos}
Let $E_1$ and $E_2$  be  smooth vector bundles over  smooth manifolds $M_1$ and $M_2$, respectively. Let $V$ be an open subset of some Euclidean space and let 
\[
\Psi_v \colon E_1 \stackrel \cong \longrightarrow E_2\qquad (v\in V)
\]
be a smooth family of vector bundle isomorphisms  that covers  a smooth family of diffeomorphisms from $M_1$ to $M_2$.  If $A$ is a compactly supported, classical pseudodifferential operator of order zero on $M_2$, acting on the sections of $E_2$, then the operators 
\[
\Psi_v  ^* A \Psi_v^{*,-1} \colon L^2 (M_1; E_1) \longrightarrow L^2 (M_1; E_1) \qquad (v \in V)
\]
vary norm-continuously with $v$.
\end{theorem}

\begin{definition} 
\label{def-continuous-family}
Let $\pi\colon  \pmb M \to W$ be a submersion of smooth manifolds with fibers $M_w$ ($w\in W)$ and let $\pmb E $ be a smooth vector bundle over $\pmb M$, restricting to bundles $E_w$ over the fibers $M_w$.  We shall say that a family of classical, pseudodifferential operators 
\[
A_w \colon C_c^\infty  (M_w;E_w) \longrightarrow C^\infty (M_w; E_w)\qquad (w\in W) 
\]
of order zero is \emph{continuous} if  
\begin{enumerate}[\rm (a)]

\item for every pair of open sets  $\pmb U\subseteq \pmb M$ and $V\subseteq W$, 

\item for every smooth, compactly supported function $\varphi $ on $\pmb{U}$,

\item for every commuting diagram
\[
\xymatrix{
L\times V \ar[r]^-{\Phi}\ar[d]_{\pi_V} & \pmb U\ar[d]^{\pi\vert_{\pmb U}}
\\
V \ar[r]_{\text{incl.}} & W
}
\]
in which $L$ is a smooth manifold, $\Phi$ is a diffeomorphism and $\pi_V$ is the projection onto the factor $V$, and    

\item for every vector bundle isomorphism
\[
\Psi \colon  L{\times} V {\times} \C^k \stackrel \cong \longrightarrow \pmb  E \vert _{\pmb U} 
\]
that covers $\Phi$, 
\end{enumerate}
the family of compactly supported, order zero, classical pseudodifferential  operators 
\[
\Psi^*_v  \varphi_v A_v \varphi_v \Psi_v^{*,-1} \colon C_c^\infty (L;\C^k ) \longrightarrow  C_c^\infty (L;\C^k )\qquad (v\in V)
\]
extends to a norm-continuous family of operators between $L^2$-spaces. Here 
\[
\Psi_v^*\colon C_c^\infty ( M_v\cap \pmb U; E) \stackrel \cong \longrightarrow 
C_c^\infty (L; \C^n)
\]
is the isomorphism of linear spaces induced from $\Psi$, and $\varphi_v$ is the restriction of $\varphi$ to $M_v\cap \pmb U$.
\end{definition} 

It follows from Theorem~\ref{thm-norm-continuity-of-psdos-under-conjgation-by-families-of-diffeos} that the continuity condition in the definition above  can be checked using a single covering family of local trivializations in the sense of (a)-(d) above, and that continuous families may be constructed using such covering families and partitions of unity. 

We are now almost ready to define the category $\PSDO^*_{\GG,\KK}$.

\begin{definition} 
We shall say that a family of equivariant operators 
\[
A _t \colon [C_c^\infty (  G_t )\otimes  V_1]^K  \longrightarrow [C_c^\infty (  G_t )\otimes  V_2]^K \qquad (t \in \R) 
\]
is a \emph{properly supported family} if all of the individual supports of the operators $A_t$, in $K \backslash G_t \times K\backslash G_t$, are included in a closed subset of $\KK\backslash \GG \times_\R\KK\backslash \GG$ for which the two coordinate projections to $\KK\backslash \GG$ are proper maps.  
\end{definition} 

\begin{definition}
\label{def-c-star-category-of-families}
    We shall denote by $\PSDO^*_{\GG,\KK}$ the $C^*$-category whose objects are the finite-dimensional unitary representations of $K$, and whose morphism spaces are the  norm-closures of the collections of all restrictions to $[0,1]$ of properly supported  families 
    \[
    A_t \colon  [C_c^\infty (  G_t )\otimes  V_1]^K  \longrightarrow [C_c^\infty (  G_t )\otimes  V_2]^K \qquad (t \in \R) 
    \]
    of equivariant, order zero, classical pseudodifferential operators that share the same principal symbol; that are smooth in the sense their compositions on the left or on the right with any smooth family of equivariant smoothing operators are again smooth families of smoothing operators; and that are norm-continuous in the sense of Definition~\ref{def-continuous-family}.  The norm here is 
    \[
    \| \{ A_t\} \| = \sup \, \{ \|A_t \|_{L^2\to L^2} : t \in [0,1]\,\}.
    \]
\end{definition}

\begin{remark} 
It follows from Proposition~\ref{prop-l-2-versus-s-star-boundedness} and the continuity condition that the norm above  is finite.
\end{remark} 

The category  $\CC^*_{\GG,\KK}$ is included in $\PSDO^*_{\GG,\KK}$  as a subcategory because if $\{ A_t\}$ is a smooth family of properly supported, equivariant smoothing operators on the fibers $G_t$ of $\pmb G$, with integral kernels $k_t(g_t,h_t)$ that vary smoothly with  $t$, $g_t$ and $h_t$, then any of the local trivializations constructed using the data (a)-(d) of Definition~\ref{def-continuous-family} is a smooth family of compactly supported smoothing operators, and so certainly varies norm-continuously with $t$.    Moreover it is included as an ideal by our smoothness assumption in Definition~\ref{def-c-star-category-of-families}, and so $\CC^*_{\GG,\KK}$ has property (i) in Section~\ref{sec-k-theory-for-psdo-category}.  

Property (ii) is built into Definition~\ref{def-c-star-category-of-families}. Surjectivity of the principal symbol functor (on morphism spaces), may be proved using the averaging construction  in \cite{Stetkaer71} and in Section~\ref{sec-psdos-on-proper-homogeneous-spaces}.  The fact that $\CC^*_{\GG,\KK}$ is the kernel of the principal symbol functor is proved similarly, by a variation of the argument in Section~\ref{subsec-psdo-extension-in-equivariant-context}.  Finally,
property (iv) is clear.

\subsection{Comments  on the van Erp-Yuncken theory and the Mackey bijection}
Erik van Erp and Robert Yuncken have presented in \cite{VanErpYuncken19} a new account of the theory of classical pseudodifferential operators that is particularly well adapted to the general perspectives of this paper. We have not used their approach here because some aspects of their theory, related to $L^2$-operator norm estimates, are not yet well-documented in the literature. But their work is so relevant to our purposes that we want to at least mention it here.

Central to the van Erp-Yuncken approach is the \emph{zoom action}
\[
\alpha_\lambda \colon \GG \longrightarrow \GG\qquad (\lambda > 0)
\]
of the group $\R^+=\{ \lambda \in \R : \lambda > 0\}$ of  positive real numbers  on the deformation space $\GG$ by diffeomorphisms. This is defined by the formula 
\[
\begin{cases}
G_t\ni g \stackrel {\alpha _\lambda} \longmapsto g\in G_{\lambda t} & t {\ne } 0
\\
G_0\ni (k,X) \stackrel {\alpha _\lambda} \longmapsto (k,\lambda ^{-1} X)\in G_{ 0} &  t{=}0 .
\end{cases}
\]
Van Erp and Yuncken do not consider the equivariant context that we have studied in this paper, but if they were to have done so, then they would have \emph{defined} equivariant order zero pseudodifferential operators as those operators on the $t{=}1$ fiber of the family of homogeneous spaces $\KK\backslash \GG$ that extend to smooth families of operators that remain fixed, modulo smooth families of smoothing operators, under the zoom action.  With this, the category $\PSDO^*_{\GG,\KK}$ that we defined above would have been built into their theory from the start---as the category whose morphism sets are generated by families that are invariant under the zoom action, modulo smooth families of smoothing operators. 

Another topic that we must at least mention is the Mackey bijection for real reductive groups \cite{Higson08,Afgoustidis19,Afgoustidis21}.  This involves the same deformation space $\pmb{G}$ that appears here, and it is also  intimately connected to the Connes-Kasparov isomorphism. Not only that, but the Mackey bijection is also  closely related to the concept of tempiric representation that will figure prominently in the coming sections. So it is natural to wonder if there might be a direct connection between pseudodifferential operators and the Mackey bijection.  Unfortunately, for the time being at least, we can shed no light on this question.

\section{A Fourier Transform for Pseudodifferential Operators}
\label{sec-fourier-iso-thm}

 The purpose of this section is to analyze the $C^*$-category $\PSDO^*_{G,K}$   
from the perspective of  tempered representation theory,. We shall construct a Fourier transform functor from $\PSDO^*_{G,K}$ to a new $C^*$-category, defined purely in terms of tempered representation theory, and   we shall prove that it is an equivalence of categories for real reductive groups $G$ of real rank one (the story for higher-real rank groups is more complicated).
 
\emph{From now on, and for the rest of the paper, we shall denote by  $G$ a real reductive group.} For definiteness we shall follow the definition of real reductive group  given by Vogan in his monograph \cite{VoganGreenBook}. As usual, we  denote by $K$ a maximal compact subgroup of $G$. We shall say that a unitary representation $\pi$  of $G$ is \emph{tempered} if the formula 
\[
\pi (f)v = \int _G f(g)\pi(g)v\, dg \qquad (f\in C_c^\infty (G),\,\, v\in H_\pi)
\]
determines a $C^*$-algebra representation of the reduced group $C^*$-algebra $C^*_r (G)$ as bounded operators on $H_\pi$, or equivalently if and only if $\pi$ is weakly contained in the regular representation of $G$.   It is shown in  \cite{CowlingHaagerupHowe} that a  unitary representation $\pi$ is tempered in this sense if and only if its $K$-finite matrix coefficient functions are tempered in the sense of Harish-Chandra. 

\subsection{Multiplicity spaces} 
Harish-Chandra showed in \cite[Thm.\,6]{HarishChandra53} that if a unitary representation $\pi$ is irreducible, then it is   \emph{admissible} in the sense that it includes each irreducible representation of $K$ with at most finite multiplicity.  Let $\pi$ be an admissible  unitary  representation  of a real reductive group $G$ on a Hilbert space $H_\pi$, and denote by $H^\infty _\pi$ the space of smooth vectors in $H_\pi$ (see for instance \cite[Ch.\ III]{KnappOverview}).  If $\Lambda$ is a compactly supported distribution on $G$, then there is unique linear   operator 
\[
\pi(\Lambda) \colon H^\infty _\pi \longrightarrow H^\infty _\pi 
\]
such that 
\[
\langle w_1 ,\pi(\Lambda) w_2\rangle = \int_G \Lambda (g) \langle w_1, \pi(g) w_2\rangle\, dg  \qquad \forall w_1\in H_\pi ,\,\, \forall w_2 \in H^\infty _\pi.
\]
Let $\mathcal{E}'(G)$ denote the space of compactly supported distributions on $G$. From the above, if $A \in [ \mathcal{E}'(G) \otimes \Hom(V_1,V_2)]^{K\times K}$, say $A = \sum_k \Lambda_k \otimes T_k$, then we obtain an operator 
\begin{equation}
    \label{eq-action-of-compactly-supported-dist-on-mult-space}
\pi(A) \colon [ H^\infty _\pi\otimes V_1]^K \longrightarrow [ H^\infty _\pi \otimes V_2] ^K
\end{equation}
by means of the formula 
\[
\pi(A) \colon  w\otimes v \longmapsto \sum_k \pi(\Lambda_k)w\otimes T_k v .
\]
The operator \eqref{eq-action-of-compactly-supported-dist-on-mult-space} may also be characterized by the formula 
\begin{multline}
    \label{eq-action-of-compactly-supported-dist-on-mult-space2}
\pi(A) (\pi(f) F) = \pi(Af)F
\\
\forall f \in [C_c^\infty(G)\otimes \End (V_1)]^{K\times K},\,\, \forall F \in [C_c^\infty (G) \otimes V_1]^K .
\end{multline}
This uses the action in \eqref{eq-action-of-smoothing-ops-in-standard-form}, plus the evident convolution action 
\begin{multline*}
[ \mathcal{E}'(G) \otimes \Hom(V_1,V_2)]^{K\times K} \times  [C_c^\infty(G)\otimes \End (V_1)]^{K\times K} 
\\
\longrightarrow  [C_c^\infty(G)\otimes \Hom (V_1,V_2)]^{K\times K}.
\end{multline*}

\begin{definition}
We shall refer to the spaces $[H^\infty _\pi \otimes V]^K$ as the \emph{multiplicity spaces} of $\pi$ (because when the $K$-representation $V$ is irreducible, the dimension of the multiplicity space is the multiplicity with which $V^*$ occurs in $H_\pi$).  Note that the multiplicity spaces of an admissible representation are finite-dimensional.
\end{definition}

Of course, all of the above applies  when $A$ is a properly supported, $G$-equivariant, order zero  classical pseudodifferential operator,
\begin{equation}
    \label{eq-psdo-for-discussion-of-fourier-transform}
A \colon [C_c^\infty (G) \otimes V_1]^K \longrightarrow [C_c^\infty (G) \otimes V_2]^K .
\end{equation} 
But we shall also need to extend the construction in \eqref{eq-action-of-compactly-supported-dist-on-mult-space} to $L^2$-operator-norm limits  order zero pseudodifferential operators.  To prepare for this, let us make the following observation:

\begin{lemma}
If $\pi$ is an admissible  representation, then   the inclusion $H^\infty _\pi\hookrightarrow  H_\pi$  induces an isomorphism of finite-dimensional vector spaces
 \[
 \bigl [ H_\pi^\infty  \otimes V\bigr ] ^K\stackrel{\cong}\longrightarrow \bigl [ H_\pi \otimes V\bigr ] ^K .
 \]
 \end{lemma}

 \begin{proof}
The two multiplicity spaces appearing in the statement of the lemma depend only on the underlying subspaces of $K$-finite vectors  in $H_\pi^\infty$ and $H_\pi$, and  these spaces coincide with one another (see for instance \cite[Prop.\ VIII.8.5]{KnappOverview}).
 \end{proof}

In the following lemma we shall use  the (completed) balanced tensor product from Hilbert $C^*$-module theory. This was introduced by Kasparov \cite[Sec.~2.8]{Kasparov80}, but see \cite{Lance95} for an exposition of the topic.

\begin{lemma}
\label{lem-balanced-kasparov-tensor-product}
    Let $\pi$ be an admissible tempered unitary representation of $G$ and let $V$ be a finite-dimensional unitary representation of $K$.  There is a unique isomorphism of Hilbert spaces
    \[
\bigl [ C_r^*(G) \otimes V\bigr ] ^K \otimes_{C^*_r(G)} H_\pi  \stackrel \cong  \longrightarrow \bigl [ H_\pi \otimes V\bigr ] ^K
\]
such that 
\[
\bigl (  \sum _j f_j \otimes v_j \bigr )  \otimes \xi \longmapsto  \sum _j \pi(f_j)\xi \otimes v_j .
\]
\end{lemma}

\begin{proof} 
The map in the statement of the lemma is well-defined on the algebraic tensor product. The image of the algebraic tensor product is dense in  $ [ H_\pi \otimes V] ^K$ (and since $ [ H_\pi \otimes V] ^K$ is finite-dimensional, the image is in fact the full target space).  In addition, it follows from the formula for the inner product on Kasparov's balanced tensor product (see \cite[(4.4)]{Lance95}) that the map is  well-defined on the balanced tensor product Hilbert space and isometric.
\end{proof}

The following definition extends the construction of $\pi(A)$ in \eqref{eq-action-of-compactly-supported-dist-on-mult-space}. 

\begin{definition}
 Let $\pi$ be an admissible tempered unitary representation of $G$, viewed as a representation of the $C^*$-algebra $C^*_r (G)$, and let 
\[
T \colon [C_c^\infty (G)\otimes V_1]^K  \longrightarrow [C_c^\infty (G)\otimes V_2]^K 
\]
be a  $G$-equivariant, properly supported, adjointable and $L^2$-bounded operator. We shall denote by 
\[
 \pi(T)  \colon 
 \bigl [ H_\pi \otimes V_1\bigr ] ^K
 \longrightarrow 
\bigl [ H_\pi \otimes V_2\bigr ] ^K
\]
the  tensor product operator
\[
 \pi(T)= T\otimes \mathrm{id}_{H_\pi} \colon 
 [C_r^*   (G)\otimes V_1]^K\otimes _{C^*_r(G)} H_\pi
 \longrightarrow 
  [C_r^*   (G)\otimes V_2]^K \otimes _{C^*_r(G)} H_\pi ,
\]
which we view an operator between multiplicity spaces using the isomorphisms in Lemma~\ref{lem-balanced-kasparov-tensor-product}. 
\end{definition}

The construction just described can be dressed up in $C^*$-category clothes, as follows. 

\begin{definition} 
We shall denote by $\Fin^*$ the $C^*$-category of finite-dimen\-sional Hilbert spaces.
\end{definition}

If $\pi$ is any admissible tempered unitary representation of $G$, then the multiplicity space construction determines a $C^*$-functor
\begin{equation}
    \label{eq-def-of-mult-pi}
\mult_\pi\colon \PSDO^*_{G,K}\longrightarrow \Fin ^*
\end{equation}
by means of the formula 
\[
\bigl [ A\colon [C^*_r (G)\otimes V_1]^K
\to [C^*_r(G)\otimes V_2]^K\bigr ] 
\stackrel{\mult_\pi}\longmapsto 
\bigl [ \pi(A)\colon [H_\pi \otimes V_1]^K
\to [H_\pi \otimes V_2]^K\bigr ] .
\]
We shall use this $C^*$-categorical perspective on the  multiplicity space construction  in Section~\ref{sec-vogan-and-tempiric-reps}. But for now, roughly speaking, we may viewing the family of all multiplicity functors as a Fourier transform for the category $\PSDO^*_{G,K}$.  We shall refine this idea in   the remainder of Section~\ref{sec-fourier-iso-thm}.

 \subsection{Multiplicity spaces for principal series representations}
 \label{subsec-multiplicity-spaces-for-principal-series}
When $\pi$ is a principal series representation, the constructions in the previous section may be made more explicit  by using an explicit model for the representation $\pi$.  The purpose of this section is to explain this for the unitary minimal principal series  (essentially the same picture emerges for all the unitary principal series representations, but in this paper we are mostly concerned with groups that are of real rank one, for which there are no higher principal series representations beyond the minimal principal series).

Let $G$ be a real reductive group and let $G=KAN$ be an Iwasawa decomposition; in addition let   $M$ be the centralizer of $A$ in $K$, so that $P=MAN$ is a minimal parabolic subgroup of $G$.  See \cite[Chs.~VI,VII]{KnappBeyond} or \cite{VoganGreenBook}.  Denote by  $\rho\in \mathfrak{a}^*$  the half-sum of the positive restricted roots associated to $N$, as in \cite{KnappBeyond} or \cite{VoganGreenBook}.   If $\sigma$ is an irreducible   unitary representation of the compact group $M$ on a finite-dimensional Hilbert space $L_\sigma$, and if $\nu \in \mathfrak{a}^*$, then the completion of the space 
\begin{multline}
\label{eq-smooth-vectors-in-principal-series}
H^\infty _{\sigma,i\nu} = \bigl \{ \,  f\colon G \stackrel{C^\infty}\to  L_\sigma \, : \, 
f(gman) = e^{-(\rho + i \nu) (\log (a)) } \sigma(m)^{-1} f(g)
\\
\forall g\in G, m\in M , a\in A, n\in N
\,\}
\end{multline}
in the norm 
\begin{equation} 
\label{eq-l-2-norm-on-k}
\| f\| ^2 = \int_K \| f(k) \|\, dk
\end{equation}
is the Hilbert space $H_{\sigma,  i\nu}$ of the unitary minimal principal series representation 
\[
\pi_{\sigma,  i \nu} = \operatorname{Ind}_P^G \, \sigma {\otimes}\! \exp( i \nu)
\]
(the action of $G$ is by left translation). As the notation suggests, $H^\infty _{\sigma, i \nu}$ is in fact the space of smooth vectors for the representation in the Hilbert space $H_{\sigma, i \nu}$.

Now let $V$ be any finite-dimensional unitary representation of $K$ and  consider  the linear map
\begin{equation}
\label{eq-multiplicity-space-as-function-space}
H_{\sigma,\nu} ^\infty\otimes V \ni  f\otimes v \longmapsto \bigl [ g\mapsto v \otimes f(g)\bigr ] \in C^\infty (G)\otimes  V\otimes L_\sigma  .
\end{equation}
It is a morphism of $C^\infty_c(G)$-modules for the left convolution action of $G$ on $C^\infty(G)$.
Keeping in mind that the multiplicity space may be computed using the smooth vectors $H^\infty _{\sigma, i \nu}\subseteq H_{\sigma, i \nu}$, we find that \eqref{eq-multiplicity-space-as-function-space} restricts to an isomorphism
\begin{multline}
\label{eq-multiplicity-space-for-principal-series}
[ H _{\sigma,i\nu}\otimes V]^K \stackrel\cong\longrightarrow   \bigl \{ \,  f\colon G \stackrel{C^\infty}\to V\otimes  L_\sigma \, : \, 
f(kgman) = e^{-(\rho + i \nu) (\log (a) }\cdot 
\\ (k\otimes \sigma(m)^{-1}) f(g)\quad 
\forall g\in G, m\in M , a\in A, n\in N
\,\bigl \}.
\end{multline}
We shall  express this more succinctly as an isomorphism 
\begin{equation}
\label{eq-multiplicity-space-as-function-space-2}
[ H _{\sigma,i\nu}\otimes V]^K \stackrel{\cong}\longrightarrow   [C^\infty( G)\otimes   V\otimes  L_\sigma ]^{K\times P} ,
\end{equation}
where the $P$-action is given by the prescription in \eqref{eq-multiplicity-space-for-principal-series}: 
\[
(p\cdot f)(g)= 
e^{(\rho + i \nu) (\log (a)) }\cdot 
\\ (\mathrm{id}_V\otimes \sigma(m) ) f(gp)\quad \forall p=man.
\]

Now let us return to our   pseudodifferential operator $A$ in \eqref{eq-psdo-for-discussion-of-fourier-transform}.  Because $A$ is properly supported, it
extends to an operator 
\[
A \colon [C^\infty (G) \otimes V_1]^K \longrightarrow [C^\infty (G)\otimes  V_2]^K,
\]
which we may tensor with the identity operator on $L_\sigma $ to obtain 
\[
A{\otimes}\operatorname{id}_{L_\sigma}  \colon [C^\infty (G) \otimes V_1\otimes L_\sigma  ]^K\longrightarrow [C^\infty (G) \otimes V_2\otimes L_\sigma ]^K
\]
(we give $L_\sigma $ the trivial action of $K$).

\begin{lemma} 
\label{lem-action-of-psdo-on-principal-series-mult-space}
The diagram 
\[
\xymatrix{
[H_{\sigma,i\nu} \otimes V_1]^K \ar[r]^{\pi_{\sigma, i\nu}(A)} \ar[d]_{\eqref{eq-multiplicity-space-as-function-space-2}} & [H_{\sigma,i\nu} \otimes V_2]^K \ar[d]^{\eqref{eq-multiplicity-space-as-function-space-2}}
\\
[C^\infty (G) \otimes V_1\otimes L_\sigma ]^{K\times P} \ar[r]_{A\otimes \mathrm{id}} & [C^\infty (G) \otimes V_2\otimes L_\sigma ]^{K\times P}
}
\]
is commutative.
\end{lemma} 

 \begin{proof}
This is a consequence of \eqref{eq-action-of-compactly-supported-dist-on-mult-space2} and the fact that the morphism \eqref{eq-multiplicity-space-as-function-space} is $C^\infty_c(G)$-linear.
 \end{proof}

\subsection{Limit formula for order zero pseudodifferential operators}
\label{subsec-limit-formula-for-psdos}
Let us continue with the same notation as the previous section.
Since $G=KP$, there are isomorphisms 
\begin{equation}
    \label{eq-frobenius-reciprocity-iso}
\xymatrix{
[H_{\sigma,i\nu}\otimes V_1]^K \ar[r]_-{\cong}^-{\eqref{eq-multiplicity-space-as-function-space-2}}&  [C^\infty (G) \otimes V_1\otimes L_\sigma ]^{K\times P} \ar[r]_-{\cong} & [V_1\otimes L_\sigma ]^M ,
}
\end{equation}
in which the second arrow is evaluation at $e\in G$.
 Using them, if $A$ is a pseudodifferential operator as in \eqref{eq-psdo-for-discussion-of-fourier-transform}, then we may regard $\pi_{\sigma, i \nu}(A)$ as an operator 
\[
\pi_{\sigma, i \nu}(A) \colon 
[V_1\otimes L_\sigma ]^M
\longrightarrow 
[V_2\otimes L_\sigma ]^M,
\]
and we may think of the \emph{family} of operators $\pi_{\sigma, i \nu}(A)$, as $\nu$ varies,  as a function
\begin{equation}
    \label{eq-family-of-maps-on-multiplicity-spaces}
\pi_\sigma (A) \colon \mathfrak{a} \longrightarrow \Hom \bigl ( [V_1\otimes L_\sigma ]^M,[V_2\otimes L_\sigma ]^M \bigr ).
\end{equation}
We wish to determine the behavior of this function at infinity in $\mathfrak {a}$.

We shall use the following  identification, coming from the Iwasawa decomposition $G=KAN$: 
\begin{equation}
\label{eq-tangent-vector-identification}
\mathfrak{a}\oplus \mathfrak{n} \stackrel \cong \longrightarrow 
\mathfrak{g} / \mathfrak{k} .
\end{equation}
This has the following geometric consequence: if we denote by $e\in K\backslash G$ the basepoint of the symmetric space associated to $G$, then the tangent space of the symmetric space   at $e$ identifies naturally with $\mathfrak{g} / \mathfrak{k}$, and hence, by \eqref{eq-tangent-vector-identification}, with $\mathfrak{a}\oplus \mathfrak{n}$.  By duality we obtain an isomorphism 
\begin{equation}
\label{eq-cotangent-vector-identification}
(\mathfrak{g} / \mathfrak{k})^*  
\stackrel{\cong}\longrightarrow \mathfrak{a}^* \oplus \mathfrak{n}^*  ,
\end{equation}
which we shall use  to identify each element $\nu \in \mathfrak{a}^*$ with a cotangent vector at $e\in K\backslash G$. 

Returning to the operator $A$ in \eqref{eq-psdo-for-discussion-of-fourier-transform}, its principal symbol, restricted to the cotangent fiber over the basepoint $e$ is a $K$-equivariant, 
$0$-homogeneous smooth function 
\[
a_0 (e,\underbar{\phantom{i}}\,) \colon  (\mathfrak{g}/\mathfrak{k})^* \setminus \{ 0\} \longrightarrow \Hom (V_1,V_2)
\]
If we further restrict to $\mathfrak{a}^*$, as above, then we obtain a $0$-homogeneous smooth function 
\[
a_0 (e,\underbar{\phantom{i}}\,) \colon  \mathfrak{a}^* \setminus \{ 0\} \longrightarrow \Hom (V_1,V_2)^M
\]
since  $\mathfrak{a}\subseteq \mathfrak{g}/\mathfrak{k}$ is centralized by $M$.  Tensoring with the identity operator on $L_\sigma $, we obtain from this a map
\begin{equation}
    \label{eq-part-of-symbol-function}
a_0 (e,\underbar{\phantom{i}}\,) \colon  \mathfrak{a}^* \setminus \{ 0\} \longrightarrow \Hom ([V_1\otimes L_\sigma ]^M,[V_2\otimes L_\sigma ]^M)
\end{equation}
(which in general characterizes part but not all of the principal symbol).

We shall now apply the compactification operation introduced in Definition~\ref{def-spherical-compactification} to   the real vector space $\mathfrak{a}^*$ to obtain from the above a continuous function 
\begin{equation}
    \label{eq-part-of-symbol-function-2}
a_0 (e,\underbar{\phantom{i}}\,) \colon  \partial \mathfrak{a}^* \longrightarrow \Hom ([V_1\otimes L_\sigma ]^M,[V_2\otimes L_\sigma ]^M),
\end{equation}
where $\partial \mathfrak{a}^*$ denotes the boundary of $\mathfrak{a}^*$ in $\overline{\mathfrak{a}^*}$.  

\begin{theorem}
\label{thm-limit-fmla-for-psdo}
Let $G= KAN$ be an Iwasawa decomposition of the real reductive group $G$, let $P=MAN$ be the corresponding minimal parabolic subgroup, and let $\sigma$ be an irreducible unitary representation of the compact group $M$ on a finite-dimensional space $L_\sigma $. Let $V_1$ and $V_2$ be  finite-dimensional unitary representations of $K$, and let $A$ be a properly supported, $G$-equivariant, order zero classical pseudodifferential operator, as in \eqref{eq-psdo-for-discussion-of-fourier-transform}.  The function in  \eqref{eq-family-of-maps-on-multiplicity-spaces}   extends to a continuous function
\[
\pi_\sigma (A) \colon \overline{\mathfrak{a}^*}\longrightarrow \Hom  ([V_1\otimes L_\sigma ]^M,[V_2\otimes L_\sigma ]^M)
\]
whose restriction to  the boundary $\partial \mathfrak{a}^* $ is the principal symbol function \eqref{eq-part-of-symbol-function-2}.
\end{theorem}

\begin{proof} 
We shall use Lemma~\ref{lem-action-of-psdo-on-principal-series-mult-space}. The functions in the version of the multiplicity space given there,  namely the space 
\[
C^\infty  (G, (V_1\otimes L_\sigma) )^{K\times P}  \cong [C^\infty (G)\otimes  V_1\otimes L_\sigma ]^{K\times P},
\]
have the form 
\[
f_{\nu,u}\colon kan\longmapsto e^{-(i\nu+\rho) (\log (a))} (k\otimes \mathrm{id}) u ,
\]
where $u=f(e) \in [V_1\otimes L_\sigma ]^M $.  The single value
\[
(Af_{\nu,u})(e) \in [C^\infty (G) \otimes V_2\otimes L_\sigma ]^{K\times P}
\]
determines $Af_{\nu,u}$, and it is a continuous function of $\nu$, for $u$ fixed,  by virtue of the continuity of pseudodifferential operators. So the function in \eqref{eq-family-of-maps-on-multiplicity-spaces} is continuous.

Since $A$ is properly supported, there is a smooth, compactly supported function $\varphi$ on $K \backslash G$  such that $\varphi (e)=1$ and 
\[
(Af_{\nu,u})(e) = (Af_{\nu,e}\varphi) (e)\qquad \forall \nu\in \mathfrak{a}^*,\,\, \forall u\in [V_1\otimes L_\sigma ]^M.
\]
If we write 
\[
(\varphi f_{\nu,u})(kan) =  e^{-i\nu (\log (a))} \cdot s(kan)  ,
\]
where 
\[
s(kan) = e^{-\rho (\log(a))} \varphi (an) \cdot (k\otimes \mathrm{id}) u , 
\]
then we may write 
\[
(Af_{\nu,u})(e) = e^{i \nu (\log(a))} (Ae^{-i \nu (\log(a))} s)(e).
\]
With this, the theorem becomes an immediate consequence of the formula \eqref{eq-limit-formula-for-principal-symbol} for the symbol of a classical pseudodifferential operator (applied to operators of order $0$).
\end{proof}

\subsection{Fourier transform isomorphism for the reduced group C*-algebra} 
\label{subsec-fourier-transform-of-c-star-algebra}
In this section we shall
give an extremely rapid  review  of the description of the reduced $C^*$-algebra of a real reductive group $G$, up to $*$-isomorphism, that may be obtained from results in tempered representation theory due to Harish-Chandra and others; for full details, see \cite{ClareCrispHigson16}.

For further brevity, \emph{we shall confine our attention to  real reductive groups $G$ of real rank one}, since our focus for the rest of  the paper will be on these groups. The real rank one condition means that the real vector space $\mathfrak{a}$ in the Iwasawa decomposition has dimension one, or equivalently that every one-dimensional subspace of the vector space $\mathfrak{s}$ in the Cartan decomposition $\mathfrak{g} = \mathfrak{k} \oplus \mathfrak{s}$ is maximal abelian.

Recall that $M$ is the centralizer of $\mathfrak{a}$ in $K$. For a fixed $\sigma\in \widehat M$, and under restriction of functions from $G$ to $K$,   the representation Hilbert spaces $H_{\sigma, i \nu}$ in the unitary minimal principal series, defined in \eqref{eq-smooth-vectors-in-principal-series} and \eqref{eq-l-2-norm-on-k},  all become isomorphic to the single Hilbert space 
\[
H_\sigma = \bigl \{ \, f \colon  K \stackrel{L^2}\to L_\sigma \,:\, f(km)= \sigma(m)^{-1} f(k)\,\,\,
\forall k\in K, \forall m\in M \, \bigr \}  .
\]
The family of principal series representations $\pi_{\sigma, i\nu}$ combine to give a single $C^*$-algebra morphism 
\[
\pi_\sigma \colon C^*_r (G)\longrightarrow C_0(\mathfrak{a}^*, \mathfrak{K} (H_\sigma)  ),
\]
for which composition with evaluation at $\nu\in \mathfrak{a}^*$ gives $\pi_{\sigma, i \nu}$ (that is, the tempered unitary representation $\pi_{\sigma, i\nu}$ viewed as a representation of $C^*_r(G)$). The range of $\pi_\sigma$ is the $C^*$-subalgebra of invariant elements for a finite group $W_\sigma$ of automorphisms of $\mathfrak{a}^*$, acting as intertwining operators on the bundle over $\mathfrak{a}^*$ with fibers $H_{\sigma,i\nu}$.  Assembling all the $\pi_\sigma$, together with the discrete series for $G$ (if they exist for $G$), one obtains a $C^*$-algebra \emph{isomorphism} 
\begin{equation}
\label{eq-fourier-transform-c-star-algebra-morphism}
C^*_r(G)  \stackrel \cong \longrightarrow \bigoplus _{[P,\sigma]}
C_0\bigl  (\mathfrak{a}^*_P, \mathfrak{K}(H_\sigma  )\bigr ) ^{W_\sigma}.
\end{equation}
The index set is the set of equivalence classes of pairs $(P,\sigma)$ consisting of 
\begin{enumerate}[\rm (i)]

\item a standard cuspidal parabolic subgroup $P$, which for real rank-one groups is either  the minimal parabolic subgroup or $G$ itself, and 

\item representations $\sigma$ which, again for real rank-one groups, are either irreducible representations of $M$, when $P$ is the minimal parabolic subgroup, or discrete series representations of $G$, when $P{=}G$.  
\end{enumerate}
There are no nontrivial  equivalences    when $P{=}G$, but when $P$ is the minimal parabolic, $[P,\sigma_1] = [P,\sigma_2]$ if and only if $\sigma_1$ and $\sigma_2$ lie in the same orbit under the action on $\widehat M$ of the restricted Weyl group $W(\mathfrak{g},
\mathfrak{a})$. The spaces $\mathfrak{a}_P$ in \eqref{eq-fourier-transform-c-star-algebra-morphism} are as follows:  $\mathfrak{a}_P {=}\mathfrak{a}$ from the Iwasawa decomposition, when $P$ is the minimal parabolic subgroup, and $\mathfrak{a}_P{=}\{ 0\}$, when $P{=}G$.  When $P$ is the minimal parabolic, the intertwining group $W_\sigma$ is the subgroup of $W(\mathfrak{g},
\mathfrak{a})$ that fixes the equivalence class of $\sigma$. 
When $P{=}G$, the intertwining group  $W_\sigma$ is trivial.  Once again, see \cite{ClareCrispHigson16} for a full account (which   applies to all real reductive groups).

\begin{theorem} 
\label{thm-uniform-admissibility}
If $V$ is a finite-dimensional unitary representation of $K$, then $[V\otimes H_\sigma]^K=0$,  for all but finitely many of the Hilbert spaces $H_\sigma$ in the direct sum decomposition \eqref{eq-fourier-transform-c-star-algebra-morphism}.
\end{theorem}

\begin{proof}
The following proof is for the real rank one case of concern to us here. Harish-Chandra proved in  \cite{HarishChandra66}  that  each irreducible representation of $K $ occurs   in at most  finitely many of the discrete series  representations  of $G$.  As for the principal series, by Frobenius reciprocity, an irreducible representation $\theta$ of $K$ occurs in $H_\sigma$ if and only if $\sigma$ is a  constituent of $\theta\vert_{M}$.  So $\theta$ occurs in at most finitely many  $H_\sigma$, as required.
\end{proof}

It follows that if $V_1$ and $V_2$ are finite-dimensional unitary representations  of $K$, then the isomorphism in \eqref{eq-fourier-transform-c-star-algebra-morphism}   induces an isomorphism
\begin{multline}
\label{eq-fourier-transform-c-star-category-morphism}
[ C_r^*(G) \otimes \operatorname{Hom}(V_1,V_2) ]^{K \times K}\stackrel \cong \longrightarrow 
\\
\bigoplus _{[P,\sigma]}
C_0\Bigl  (\mathfrak{a}^*_P, \mathfrak{K}\bigl ([V_1\otimes H_\sigma]^K,[V_2\otimes H_\sigma]^K\bigr )\Bigr ) ^{W_\sigma},
\end{multline}
in which the direct sum on the right-hand side is finite.

\subsection{Definition of the Fourier transform for order zero operators}
If we define a $C^*$-category with the usual objects---the finite-dimensional unitary representations of $K$---and with morphism spaces given by the  right-hand side  of \eqref{eq-fourier-transform-c-star-category-morphism}, then in view of  \eqref{eq-norm-completion-f-smoothing-ops}, the isomorphism  \eqref{eq-fourier-transform-c-star-category-morphism} can be viewed as a \emph{Fourier transform isomorphism}   from $\CC_{G,K}^*$ to this $C^*$-category.  Our goal in this section is to obtain a similar isomorphism for  $\PSDO^*_{G,K}$, at least when $G$ has real rank one.

The following theorem, which defines our Fourier transform functor for the category $\PSDO^*_{G,K}$, is an immediate consequence of the limit formula in Theorem~\ref{thm-limit-fmla-for-psdo}.

\begin{theorem}
\label{thm-unique-extension-of-fourier-transform}
There is a unique extension of the Fourier transform functor 
\[
\CC^*_{G,K}(V_1,V_2) \stackrel \cong \longrightarrow \bigoplus _{[P,\sigma]}
C_0\bigl  (\mathfrak{a}^*_P, \mathfrak{K}([V_1\otimes H_\sigma]^K,[V_2\otimes H_\sigma]^K)\bigr ) ^{W_\sigma}
\]
in \eqref{eq-fourier-transform-c-star-category-morphism} to a faithful functor 
\[
\PSDO^*_{G,K} (V_1, V_2) \longrightarrow \bigoplus_{[P,\sigma]} C \bigl (\overline{\mathfrak{a}^*_P}, \Compact ([V_1\otimes H_\sigma]^K,[V_2\otimes H_\sigma]^K )\bigr )^{W_\sigma} 
\]
\textup{(}as  above, we regard the right-hand sides in the displays as morphism spaces in   $C^*$-categories with objects the finite-dimensional unitary representations of $K$\textup{)}. \qed
\end{theorem}

To emphasize the obvious,   the domain of the extended Fourier transform functor in the theorem is constructed  using only pseudodifferential operator theory on $K \backslash G$, while the range is constructed  using only the tempered representation theory of $G$.

\begin{theorem}
\label{thm-fourier-isomorphism-in-real-rank-one}
  If $G$ has real rank one, then the morphisms in Theorem~\textup{\ref{thm-unique-extension-of-fourier-transform}} are isomorphisms.  
\end{theorem}

\begin{proof} Let us write 
\[
\Compact _{\sigma, V_1,V_2} =\Compact  ([H_\sigma \otimes V_1]^K,[H_\sigma \otimes V_2]^K),
\]
and then write 
\[
A_{P,\sigma} =  C  (\overline{\mathfrak{a}^*_P},  \Compact _{\sigma, V_1,V_2}  )  )^{W_\sigma} 
\quad \text{and} \quad 
J_{P,\sigma} =  C_0  ({\mathfrak{a}^*_P}, \Compact _{\sigma, V_1,V_2}  )^{W_\sigma}  .
\]
We need to show that the middle vertical map in the diagram
\begin{equation*}
\xymatrix@C=20pt{
0 \ar[r] & \CC^*_{G,K} (V_1,V_2)  \ar[r]\ar[d] & \PSDO^*_{G,K} (V_1,V_2) \ar[d] \ar[r] & \Symb^{*}_{G,K}(V_1, V_2)  \ar[r]\ar[d] & 0
\\
    0 \ar[r] & \bigoplus_{[P,\sigma]}   J_{P,\sigma}  \ar[r] &  \bigoplus_{[P,\sigma]} A_{P,\sigma}
\ar[r] &  \bigoplus_{[P,\sigma]} A_{P,\sigma}/J_{P,\sigma}  \ar[r] & 0
}
\end{equation*}
is an isomorphism.    According to the isomorphism in \eqref{eq-fourier-transform-c-star-category-morphism}, the left-hand vertical map is an isomorphism, and according to Theorem~\ref{thm-unique-extension-of-fourier-transform} the middle vertical map, and hence also right right vertical map, are injective. We shall prove that the right vertical map is an isomorphism by showing that the dimensions of its range and target spaces are finite and equal; this will also prove that the middle vertical map is an isomorphism.

We shall not consider the elementary and degenerate case where $G$ is the Cartesian product of a (one-dimensional) vector group and a compact group.  In the other cases where $G$ has real rank one, the action of $K$ on $\partial (\mathfrak{g}/\mathfrak{k})^*$ is transitive. This is because \begin{enumerate}[\rm (i)]
    \item any two maximal abelian subspaces of the space $\mathfrak{s}$ in the Cartan decomposition $\mathfrak{g} = \mathfrak{k}\oplus \mathfrak{s}$ are conjugate under $K$ (see for instance \cite[Thm.\ 6.51]{KnappBeyond}), and, as already noted, under the real rank one hypothesis, any one-dimensional subspace is maximal abelian, and 

    \item in the nondegenerate case, where $\mathfrak{a}$ is not central in $\mathfrak{g}$, there is at least one restricted root, and hence there is an element $w$ in the normalizer of $\mathfrak{a}$ in $K$ that acts as multiplication by $-1$ on $\mathfrak{a}$ (see for instance \cite[Thm.\ 6.57]{KnappBeyond}). So both of the unit vectors in $\mathfrak{a}$ lie in the same $K$-orbit in $\mathfrak{g}$, and therefore, thanks to (i), \emph{all} unit vectors in $\mathfrak{g}$  lie in the same $K$-orbit. 
\end{enumerate} 
Note that outside of the elementary and degenerate case, $W(\mathfrak{g},\mathfrak{a})=\Z_2$, generated by the element $w$ in (ii).

The isotropy group for the action of $K$ on $\partial (\mathfrak{g}/\mathfrak{k})^*$ is $M$, so that 
\[
\Symb^{*}_{G,K}(V_1, V_2)  = C(\partial (\mathfrak{g}/\mathfrak{k})^*, \Hom (V_1,V_2))^K \cong \Hom (V_1,V_2)^M. 
\]
On the other hand 
\begin{equation*}
\bigoplus _{[P,\sigma]} A_{P,\sigma}/J_{P,\sigma} \cong \bigoplus _{[P,\sigma]} C (\partial{\mathfrak{a}^*_P},   \Compact _{\sigma, V_1,V_2}  )^{W_\sigma}  
\cong 
 \bigoplus _{\sigma \in \widehat M} 
\Compact _{\sigma, V_1,V_2}.
\end{equation*}
The second isomorphism may be explained as follows.
\begin{enumerate}[\rm (i)]

\item When $P{=}G$, the space $\partial{\mathfrak{a}^*_P}$ is the empty set, and the  direct summand indexed by any $[P,\sigma]$ is zero.

\item When $P$ is the minimal parabolic in $G$, and when $W_\sigma$ is the two-element group $W(\mathfrak{g},\mathfrak{a})$, that group acts freely and transitively on $\partial{\mathfrak{a}^*_P}$, but it fixes the equivalence class of $\sigma$. The direct summand indexed by $[P,\sigma]$ is  $ \Compact _{\sigma, V_1,V_2}$. 

\item When $P$ is the minimal parabolic in $G$, and when $W_\sigma$ is   trivial, $\partial{\mathfrak{a}^*_P}$ is again a two-point space,  and the  summand indexed by $[P,\sigma]$ is the  sum of two copies of  $\Compact _{\sigma, V_1,V_2}$.  But if $w$ is the generator $w\in W(\mathfrak{g},\mathfrak{a})$, then  $[P,\sigma]= [P,w(\sigma)]$ and $ \Compact _{\sigma, V_1,V_2}\cong  \Compact _{w(\sigma), V_1,V_2}$.

\end{enumerate}
It follows from the isomorphism and   \eqref{eq-frobenius-reciprocity-iso} that 
\begin{equation*}
\sum_{[P,\sigma]} \dim_{\C} \bigl (  A_{P,
\sigma}/J_{P,\sigma} \bigr )
= 
\sum_{\sigma \in \widehat M}
\dim \bigl ( [ L_{\sigma} \otimes V_1  ]^M \bigr )\cdot \dim \bigl ( [ L_{\sigma} \otimes V_2  ]^M \bigr ) .
\end{equation*}
But by taking  the $M$-isotypical decompositions of $V_1$ and $V_2$, we find that 
\[
\dim \bigl ( \Hom(V_1,V_2)^M \bigr ) =\sum_{\sigma \in \widehat M}
\dim \bigl ( [ L_{\sigma} \otimes V_1  ]^M \bigr )\cdot \dim \bigl ( [ L_{\sigma} \otimes V_2  ]^M \bigr ).  
\]
This proves the theorem. 
\end{proof}

\section{Tempiric Representations and Vogan's Theorem}
\label{sec-vogan-and-tempiric-reps} 

In this final section we shall connect the  $K$-theory and   Fourier theory of the $C^*$-category $\PSDO^*_{G,K}$  to  David Vogan's theorem on   the minimal $K$-types of tempiric representations.

\subsection{Tempiric representations from the C*-category point of view} 

\begin{definition}
    An irreducible  unitary representation of a real reductive group is \emph{tempiric} if it is tempered, irreducible,  and has real infinitesimal character.  
\end{definition}

See \cite[Sec.\ 2]{Vogan00} for a discussion from first principles of the concept of real infinitesimal character. For our present purposes, we define the tempiric representations of $G$ to be   precisely those that occur as summands in representations with continuous parameter $0\in \mathfrak{a}^*_P$ in the description \eqref{eq-fourier-transform-c-star-algebra-morphism} of $C^*_r(G)$; see \cite[Cor.\ 3.5]{Vogan00}.

\begin{definition} 
Let $G$ be a real reductive group.  
We shall denote by  $R(G)_{\tempiric}$ the free abelian group on the set of unitary equivalence classes of tempiric representations of $G$.
\end{definition}

\begin{definition} 
\label{def-rep-tempiric-category}
Define the $C^*$-category $\Rep^*_{G,K,\tempiric }$ as follows: 
\begin{enumerate}[\rm (i)]

\item The objects are the finite-dimensional unitary representations of $K$.

\item The morphisms are  families of morphisms between multiplicity spaces, 
\[
 T_\pi  \colon [H_\pi \otimes V_1 ]^K \longrightarrow [H_\pi \otimes V_2 ]^K \qquad ([\pi]\in \widehat G _{\tempiric}),
\]
indexed by representatives of the unitary  equivalence classes of tempiric representations.
\end{enumerate}
\end{definition}

Let $\pi$ be a tempiric representation. We define a  functor 
\[
\proj_\pi \colon \Rep^*_{G,K,\tempiric } 
\longrightarrow \Fin^*
\]
by mapping each object $V$ of $\Rep^*_{G,K,\tempiric }$, that is, each finite-dimensional unitary representation of $K$, to the finite-dimensional Hilbert space $[H_\pi \otimes V]^K$, and   each morphism, as in part (ii) of Definition~\ref{def-rep-tempiric-category}, to its $\pi$-component. Then, passing  to  $K_0$-groups, we obtain  a morphism
\[
 \proj_{\pi,0} \colon K_0\bigl ( \Rep^*_{G,K,\tempiric }\bigr )  \longrightarrow  \Z.
\]
since $K_0(\Fin^*) = \Z$.

\begin{lemma} 
\label{lem-tempiric-rep-ring-as-k-theory-group}
The $K$-theory morphisms $\proj_{\pi,0}$, indexed by a set of representatives of the unitary equivalence classes of tempiric representations of $G$, induce an isomorphism of abelian groups
\[
K_0 ( \Rep^*_{G,K,\tempiric } )  \ni c   \longmapsto \sum _{\pi} \proj_{\pi,0}(c) [\pi] \in R(G)_{\tempiric } .
\]
In addition, $K_1  ( \Rep^*_{G,K,\tempiric }  ) = 0$.
\end{lemma}

\begin{proof}
It follows from Theorem~\ref{thm-uniform-admissibility} (because each representation $\pi_{\sigma,0}$  on $H_\sigma$ decomposes into at most finitely many tempiric summands) that if $V$ is any finite-dimensional unitary representation  of $K$, then   $[H_\pi \otimes V ]^K=0$ for all but  finitely many tempiric $\pi$ (up to unitary equivalence, of course).   It follows that if $c\in K_0  ( \Rep^*_{G,K,\tempiric }  ) $, then $\proj_{\pi,0}(c) = 0$ for all but finitely many $\pi$.  So the map in the statement of the theorem is well-defined.

Fix for a moment a single tempiric representation $\pi$, and consider the additive $C^*$-category whose objects are finite-dimensional unitary representations  of $K$ and whose morphisms are linear maps $[H_\pi \otimes V_1 ]^K \longrightarrow [H_\pi \otimes V_2 ]^K$. The  natural functor (defined like $\proj_\pi$ above) from this category to $\Fin^*$ is an isomoprhism on $K$-theory. Now the category $\Rep^*_{G,K,\tempiric }$ is the direct sum of these categories in the sense of \cite[Def.\,3.8]{MitchenerCategories02}, over all $\pi$, and $K$-theory commutes with direct sums.
\end{proof} 

\begin{definition} We shall denote by   
\[
\mult \colon \PSDO_{G,K}^* \longrightarrow  \Rep^*_{G,K,\tempiric }
\]
the unique functor such that for every tempiric $\pi$, there is a commuting  diagram 
\[
\xymatrix{
\PSDO_{G,K}^* \ar[r]^-{\mult}\ar[d]_{\mult_\pi} & \Rep^*_{G,K,\tempiric }\ar[d]^-{\proj_\pi}  
\\ \Fin^* \ar@{=}[r] &\Fin^* .
}
\]
\end{definition}

\begin{lemma}
If $G$ is any real reductive group, then the functor 
\[
\mult \colon \PSDO^*_{G,K} \longrightarrow \Rep^*_{G,K,\tempiric }
\]
is surjective on $\Hom$-spaces \textup{(}it is the identity on objects\textup{)}. 
\end{lemma}

\begin{proof} 
In fact the restriction of the functor to the ideal $\CC^*_{G,K}$, namely the functor 
\[
\mult \colon \CC^*_{G,K} \longrightarrow \Rep^*_{G,K,\tempiric }, 
\]
is already surjective.  This is a consequence of the general isomorphism theorem for $C^*_r(G)$ that is proved in \cite{ClareCrispHigson16}, and outlined in Section~\ref{subsec-fourier-transform-of-c-star-algebra} for groups of real rank one (which are the only groups for which we shall use this lemma below). In the notation used in the isomorphism 
\begin{multline*}
[ C_r^*(G) \otimes \operatorname{Hom}(V_1,V_2) ]^{K \times K}\stackrel \cong \longrightarrow 
\\
\bigoplus _{[P,\sigma]}
C_0\Bigl  (\mathfrak{a}^*_P, \mathfrak{K}\bigl ([V_1\otimes H_\sigma]^K,[V_2\otimes H_\sigma]^K\bigr )\Bigr ) ^{W_\sigma}
\end{multline*}
from \eqref{eq-fourier-transform-c-star-category-morphism}, the functor $\mult$ above corresponds to evaluation in each summand on the right-hand side at the point $0\in \mathfrak {a}^*_P$. But this evaluation functor is evidently surjective.
\end{proof} 

We have, then, an extension of $C^*$-categories
\[
0\longrightarrow 
\mathsf{Kernel}^*_{G,K} \longrightarrow 
\PSDO^*_{G,K} \longrightarrow \Rep^*_{G,K,\tempiric }
\longrightarrow 0 .
\]
The main result of this section is as  follows:

\begin{theorem}
    \label{thm-homotopy-equivalence}
If $G$ is a real reductive group of real rank one, then the kernel of the functor 
\[
\mult \colon \PSDO^*_{G,K} \longrightarrow \Rep^*_{G,K,\tempiric }
\]
is a contractible ideal in $\PSDO^*_{G,K}$. 
\end{theorem}

Once again, \cite[Sec.\,3]{MitchenerKTheory01}  for the notion of contractibility, which implies vanishing of $K$-theory. We are using the term \emph{kernel} in the sense described in Section~\ref{sec-c-star-categories-from-psdo}. 

\begin{proof}
We shall use Theorem~\ref{thm-unique-extension-of-fourier-transform}.   The kernel is a direct sum, indexed by the set of all associate classes $[P,\sigma]$,  of $C^*$-categories  with morphism spaces 
\[
 C_0 \bigl (\overline{\mathfrak{a}^*_P}\setminus \{ 0\}, \Compact ([V_1\otimes H_\sigma]^K,[V_2\otimes H_\sigma]^K )\bigr )^{W_\sigma} 
\]
(in every case the objects are the finite-dimensional unitary representations of $K$), and it suffices to show that each is contractible.  There are three cases to consider:

\smallskip

\noindent \emph{Case 1: $\mathfrak{a}^*_P=0$.}  This is the discrete series case, and  the   
$C^*$-category direct summand associated above to $[P,\sigma]$  has only  zero morphisms.

\smallskip

\noindent \emph{Case 2: $\dim \mathfrak{a}^*_P=1$ and $W_\sigma$ is trivial.} Fix a homeomorphism $[-1,1]\cong \overline{\mathfrak{a}_P^*}$ mapping $0$ to $0$. The functors 
\begin{multline*}
F_t \colon  C_0 \bigl ([-1,1]{\setminus} \{ 0\}, \Compact ([V_1\otimes H_\sigma]^K,[V_2\otimes H_\sigma]^K )\bigr ) 
\\
\longrightarrow  C_0 \bigl ([-1,1]{\setminus} \{ 0\}, \Compact ([V_1\otimes H_\sigma]^K,[V_2\otimes H_\sigma]^K )\bigr )  \qquad (t\in [0,1])
\end{multline*}
defined by $F_t(f) (\nu) = f(t \nu)$ 
(the functors  are the identity on objects) constitute a homotopy from the identity functor at $t{=}1$ to the zero functor at $t{=}0$.
\smallskip

\noindent \emph{Case 3: $\dim \mathfrak{a}^*_P=1$ and $W_\sigma$ is nontrivial.} In this case, if $\mathfrak{a}^*_{P,+}\subseteq \mathfrak{a}^*_{P}$ is a half-line fundamental domain for the action of $W_\sigma {\cong} \Z_2$, then our $C^*$-category is isomorphic to the $C^*$-category with morphism spaces 
\[
 C_0 \bigl (\overline{\mathfrak{a}^*_{P,+}}\setminus \{ 0\}, \Compact ([V_1\otimes H_\sigma]^K,[V_2\otimes H_\sigma]^K )\bigr )  
\]
(where the overbar indicates the one-point compactification of the half-line $\mathfrak{a}^*_{P,+}$) and the same formula as in the previous case defines a homotopy from the identity functor to the zero functor.
\end{proof}

\begin{corollary} 
\label{cor-k-theory-and-tempirics}
If $G$ is a real reductive group of real rank one, then 
the functor 
\[
\mult \colon \PSDO^*_{G,K} \longrightarrow \Rep^*_{G,K,\tempiric }
\]
induces an isomorphism in $K$-theory: 
\[
\pushQED{\qed}
K_*\bigl (\PSDO^*_{G,K}\bigr ) \stackrel\cong \longrightarrow K_* \bigl (  \Rep^*_{G,K,\tempiric }  \bigr  ) .
\qedhere
\popQED
\]
\end{corollary} 

Putting together the corollary and Lemma~\ref{lem-tempiric-rep-ring-as-k-theory-group} we obtain in degree zero an isomorphism
\begin{equation}
\label{eq-k-theory-and-tempirics}
\mult \colon K_0( \PSDO^*_{G,K})  \stackrel \cong \longrightarrow R(G)_{\tempiric } 
\end{equation}
for real reductive groups with real rank one.  We should emphasize that we have obtained this from the Fourier isomorphism  for the $C^*$-category $\PSDO_{G,K}$ in Theorem~\ref{thm-fourier-isomorphism-in-real-rank-one}, rather than from the Connes-Kasparov isomorphism that we considered earlier the paper. But by combining \eqref{eq-k-theory-and-tempirics} with the Connes-Kasparov isomorphism we obtain a third isomorphism that is expressible in simple, representation-theoretic terms (we are using the standard identification of $K_0(\Rep^*_K) $ with $R(K)$): 

\begin{lemma}
\label{lem-composition-of-ck-and-mult}
Let $G$ be a real reductive group of real rank one. 
The composite isomorphism
\[
\xymatrix@C=30pt{
R(K)  \ar[r]^-{\CoKa}_-{\cong} &
K_0(\PSDO^*_{G,K})  \ar[r]^-{\mult}_-{\cong} & R(G)_{\tempiric }
}
\]
maps the $K$-theory class of an irreducible unitary representation $\tau\in \widehat K$ to the element 
\[
\sum_\pi \operatorname{mult} (\tau^* ,\pi)\cdot [\pi]  \in  R(G)_{\tempiric }.
\]
Here $\operatorname{mult}(\tau^*, \pi)\in \Z$ is the multiplicity with which $\tau^*$ occurs in \textup{(}the restriction to $K$ of\textup{)} $\pi$.  
\end{lemma}

\begin{proof} 
The formula in the lemma is a consequence of the formula
\[
\mult_\pi\circ \CoKa \colon \bigl ( V\stackrel {\mathrm{id}} \longrightarrow V\bigr )  \longmapsto \bigl ( [H_\pi \otimes V]^K \stackrel {\mathrm{id}}\longrightarrow [H_\pi \otimes V]^K  \bigr )  
\]
involving the composition of the functors $\mult_\pi$ and $\CoKa$. When $V$ is irreducible, the dimension of $[H_\pi \otimes V]^K $ is the multiplicity with which $V^*$ occurs in $H_\pi$.
\end{proof}
\begin{remark}
    The isomorphism in the lemma above is taken much further in \cite{BraddHigsonYuncken24}, again using the Connes-Kasparov isomorphism. That paper follows a different argument that does not involve pseudodifferential operators, but the present work was part of the motivation for it.
\end{remark}

\subsection{Vogan's theorem}
The results in the previous section may be compared with the following celebrated theorem of David Vogan (see Section~1 for references).  The theorem uses the notion of \emph{minimal $K$-type} of a representation, for which we refer the reader to \cite{VoganGreenBook} (the precise definition is not important here).

\begin{theorem}[Vogan]
\label{thm-vogan-min-k-type-bijection}
Every tempiric representation of $G$ has a unique minimal $K$-type, and every $K$-type occurs as a minimal $K$-type  in a unique tempiric representation of $G$, up to unitary equivalence; moreover it occurs there with multiplicity one.
\end{theorem}

Vogan's theorem readily implies that the composition in Lemma~\ref{lem-composition-of-ck-and-mult} is an isomorphism, since if we write the composition  in matrix form,  using Vogan's theorem to index both rows and columns by elements of $\widehat K$, then the matrix is lower triangular with diagonal entries $1$, and is hence invertible.  

\begin{remark}
Naturally, it would be very interesting to understand better the distance between Lemma~\ref{lem-composition-of-ck-and-mult} and Theorem \ref{thm-vogan-min-k-type-bijection}, and whether other techniques from $K$-theory could bridge between them.  For instance, in \cite{Lafforgue02ICM}, Lafforgue used Kasparov's dual-Dirac method in operator $K$-theory \cite{Kasparov88} to recover Harish-Chandra's classification of the discrete series from the Connes-Kasparov method, and it is natural to wonder whether a similar technique could be used here. 
\end{remark}

\begin{remark}
It appears to be an interesting  problem to formulate some version of our Fourier isomorphism theorem  for $\PSDO^*_{G,K}$ beyond real rank one. Our isomorphism theorem (or rather the verbatim extension of it beyond real rank one) fails for the product $SL(2,\R){\times} SL(2,\R)$, basically because tensor products of pseudodifferential operators are not necessarily pseudodifferential, even when one of the factors is a smoothing operator.  This can be likely  
remedied by using variants of pseudodifferential operators that are adapted to products, as in \cite{Rodino75}.  But it is a  challenge to go  further.
\end{remark}

\bibliographystyle{alpha}
\bibliography{Refs}

\end{document}